\documentclass[11pt]{article}
\usepackage{latexsym,amsfonts,amssymb,amsmath,amsthm}
\usepackage{graphicx}

\usepackage[usenames,dvipsnames]{color}
\usepackage{ulem}

\begin{document}
\setlength{\baselineskip}{16pt}

\parindent 0.5cm
\evensidemargin 0cm \oddsidemargin 0cm \topmargin 0cm \textheight
22cm \textwidth 16cm \footskip 2cm \headsep 0cm

\newtheorem{theorem}{Theorem}[section]
\newtheorem{lemma}{Lemma}[section]
\newtheorem{proposition}{Proposition}[section]
\newtheorem{definition}{Definition}[section]
\newtheorem{example}{Example}[section]
\newtheorem{corollary}{Corollary}[section]

\newtheorem{remark}{Remark}[section]

\numberwithin{equation}{section}

\def\p{\partial}
\def\I{\textit}
\def\R{\mathbb R}
\def\C{\mathbb C}
\def\u{\underline}
\def\l{\lambda}
\def\a{\alpha}
\def\O{\Omega}
\def\e{\epsilon}
\def\ls{\lambda^*}
\def\D{\displaystyle}
\def\wyx{ \frac{w(y,t)}{w(x,t)}}
\def\imp{\Rightarrow}
\def\tE{\tilde E}
\def\tX{\tilde X}
\def\tH{\tilde H}
\def\tu{\tilde u}
\def\d{\mathcal D}
\def\aa{\mathcal A}
\def\DH{\mathcal D(\tH)}
\def\bE{\bar E}
\def\bH{\bar H}
\def\M{\mathcal M}
\renewcommand{\labelenumi}{(\arabic{enumi})}

\def\disp{\displaystyle}
\def\undertex#1{$\underline{\hbox{#1}}$}
\def\card{\mathop{\hbox{card}}}
\def\sgn{\mathop{\hbox{sgn}}}
\def\exp{\mathop{\hbox{exp}}}
\def\OFP{(\Omega,{\cal F},\PP)}
\newcommand\JM{Mierczy\'nski}
\newcommand\RR{\ensuremath{\mathbb{R}}}
\newcommand\CC{\ensuremath{\mathbb{C}}}
\newcommand\QQ{\ensuremath{\mathbb{Q}}}
\newcommand\ZZ{\ensuremath{\mathbb{Z}}}
\newcommand\NN{\ensuremath{\mathbb{N}}}
\newcommand\PP{\ensuremath{\mathbb{P}}}
\newcommand\abs[1]{\ensuremath{\lvert#1\rvert}}

\newcommand\normf[1]{\ensuremath{\lVert#1\rVert_{f}}}
\newcommand\normfRb[1]{\ensuremath{\lVert#1\rVert_{f,R_b}}}
\newcommand\normfRbone[1]{\ensuremath{\lVert#1\rVert_{f, R_{b_1}}}}
\newcommand\normfRbtwo[1]{\ensuremath{\lVert#1\rVert_{f,R_{b_2}}}}
\newcommand\normtwo[1]{\ensuremath{\lVert#1\rVert_{2}}}
\newcommand\norminfty[1]{\ensuremath{\lVert#1\rVert_{\infty}}}

\title{Effects of localized spatial variations on the uniform persistence and spreading speeds of time periodic two species competition systems}%\thanks{Partially supported by NSF grant DMS--0907752}}

\author{Liang Kong and Tung Nguyen\\
Department of Mathematical Sciences\\
University of Illinois at Springfield\\
Springfield, IL 62703\\
\\
and\\
\\
Wenxian Shen\\
Department of Mathematics and Statistics\\
Auburn University\\
Auburn University, AL 36849\\
U.S.A.}

\date{}
\maketitle
\noindent{\bf Abstract.} The current paper is devoted to the study of two species competition systems of the form
\begin{equation*}
\begin{cases}
u_t(t,x)= \mathcal{A} u+u(a_1(t,x)-b_1(t,x)u-c_1(t,x)v),\quad x\in\RR\cr
v_t(t,x)= \mathcal{A} v+ v(a_2(t,x)-b_2(t,x)u-c_2(t,x) v),\quad x\in\RR
\end{cases}
\end{equation*}
where
$\mathcal{A}u=u_{xx}$,  or $(\mathcal{A}u)(t,x)=\int_{\RR}\kappa(y-x)u(t,y)dy-u(t,x)$
($\kappa(\cdot)$ is a smooth non-negative
 convolution kernel supported on an interval centered at the origin),  $a_i(t+T,x)=a_i(t,x)$, $b_i(t+T,x)=b_i(t,x)$, $c_i(t+T,x)=c_i(t,x)$, and
 $a_i$, $b_i$, and $c_i$ ($i=1,2$) are spatially homogeneous when $|x|\gg 1$, that is,
$a_i(t,x)=a_i^0(t)$, $b_i(t,x)=b_i^0(t)$, $c_i(t,x)=c_i^0(t)$ for some $a_i^0(t)$, $b_i^0(t)$, $c_i^0(t)$, and $|x|\gg 1$.
Such a system can be viewed as a time periodic competition system subject to certain localized spatial variations.
We, in particular,  study the effects of localized spatial variations on the uniform persistence and spreading speeds of the system.
Among others, it is proved that any localized spatial variation   does not affect  the
uniform persistence of the system,   does not slow down the spreading speeds of the system, and
under some linear determinant condition,   does not speed up the spreading speeds.

\medskip

\medskip


\section{Introduction}

In this paper, we consider the effects of  localized spatial variations on the asymptotic
dynamics of the following two species competition system,
\begin{equation}
\label{main-eq1}
\begin{cases}
u_t(t,x)= \mathcal{A} u+u(a_1^0(t)-b_1^0(t)u-c_1^0(t)v),\quad x\in\RR\cr
v_t(t,x)= \mathcal{A} v+ v(a_2^0(t)-b_2^0(t)u-c_2^0(t) v),\quad x\in\RR,
\end{cases}
\end{equation}
where
$\mathcal{A}u=u_{xx}$, which is referred to as a {\it random dispersal operator}, or $(\mathcal{A}u)(t,x)=\int_{\RR}\kappa(y-x)u(t,y)dy-u(t,x)$
($\kappa(\cdot)$ is a smooth non-negative
 convolution kernel supported on an interval centered at the origin and $\int_{\RR}\kappa(z)dz=1$), which is referred to as a {\it nonlocal dispersal operator},
and $a_i^0,b_i^0,c_i^0$ $(i=1,2)$ are positive H\"older continuous periodic functions with period $T$.
To be more precise, consider
 \begin{equation}
\label{main-eq2}
\begin{cases}
u_t(t,x)= \mathcal{A} u+u(a_1(t,x)-b_1(t,x)u-c_1(t,x)v),\quad x\in\RR\cr
v_t(t,x)= \mathcal{A} v+ v(a_2(t,x)-b_2(t,x)u-c_2(t,x) v),\quad x\in\RR,
\end{cases}
\end{equation}
where $a_i(t+T,x)=a_i(t,x)$, $b_i(t+T,x)=b_i(t,x)>0$, $c_i(t+T,x)=c_i(t,x)>0$, and
$a_i(t,x)=a_i^0(t)$, $b_i(t,x)=b_i^0(t)$, $c_i(t,x)=c_i^0(t)$ for $|x|\gg 1$.
System \eqref{main-eq2} can be viewed as a localized spatially perturbed system of \eqref{main-eq2}.
We study the asymptotic dynamics of \eqref{main-eq2}, including the persistence, coexistence, and invasion speeds,
and investigate the similarities and differences between \eqref{main-eq1} and \eqref{main-eq2}, which reflect the effects of localized spatial variation on the asymptotic dynamics of \eqref{main-eq1}.

Systems \eqref{main-eq1} and \eqref{main-eq2}  are used to model the population dynamics of two competing species.
Various  temporal and spatial variations exist in almost all real world problems.
System \eqref{main-eq1} takes certain seasonal temporal variation of the underlying environment into consideration. System \eqref{main-eq2}
further takes some localized special variation of the underlying environment into consideration.
Important dynamical issues about \eqref{main-eq1} and \eqref{main-eq2} include persistence, coexistence, extinction, spatial spreading, etc.
They have been studied in many papers and are quite well understood for \eqref{main-eq1}.
See \cite{Bao2013,Conley1984,Dunbar1983,Fang2014,Fang2015,Guo2011,HeSh,Hosono1998,Huang2010,Kan1995,Kan1997,LeLiWe,YuZh,Zhao2011}, etc.  for the case that $\mathcal{A} u= u_{xx}$ and see \cite{Bao2015,Fang2014,HeNgSh,KoRaSh,LZZ2015,NgRa}, etc.  for the case $\mathcal{A}u(t,x)=\int_{\RR} \kappa(y-x)u(t,y)dy-u(t,x)$.
But the study on the asymptotic dynamics of two species competition systems in  both temporal and spatial heterogeneous environments is not much.
 The objective of the current paper is to study  the asymptotic dynamics of \eqref{main-eq2},
 in particular, the effects of the localized spatial variation on the asymptotic dynamics of \eqref{main-eq1}. The reader is referred to \cite{BeHaRo, KoSh, KoSh1} for the study of localized spatial variations on the asymptotic dynamics of one species population models.

To roughly state the problems to be studied in this paper and the main results of this paper, we first give a brief review about the asymptotic dynamics
of \eqref{main-eq1}. To this end,
set
\begin{equation}
\label{coefficient-eq}
\begin{cases}
a_{iL(M)}^0=\inf_{t\in\RR} (\sup_{t\in\RR})a_i^0(t)\cr b_{iL(M)}^0=\inf_{t\in\RR} (\sup_{t\in\RR})b_i^0(t)\cr c_{iL(M)}^0=\inf_{t\in\RR} (\sup_{t\in\RR})c_i^0(t)
\end{cases}
\end{equation}
for $i=1,2$. Let (H0) be the following standing assumption.

\medskip

\noindent {\bf (H0)} {\it $a_{iL}^0>0$, $b_{iL}^0>0$, and $c_{iL}^0>0$ for $i=1,2$.}

\medskip

\noindent Throughout this paper, we assume (H0). The following results on the persistence, coexistence, and extinction in \eqref{main-eq1}  have been proved in literature
(see, for example, \cite{HeSh},  \cite{KoRaSh}, \cite{NgRa}).

\begin{itemize}
\item[$\bullet$] (semi-trivial solutions) {\it Then \eqref{main-eq1} has  two spatially homogeneous
 and time  $T$-periodic semitrivial solutions $(u_0^*(t),0)$ and $(0,v_0^*(t))$ $(\inf_{t\in\RR}u_0^*(t)>0$ and $\inf_{t\in\RR}v_0^*(t)>0)$.
}

\item[$\bullet$] (Persistence and coexistence)
{\it   If $a_{1L}^0>\frac{c_{1M}^0a_{2M}^0}{c_{2L}^0}$ and
$a_{2L}^0>\frac{b_{2M}^0a_{1M}^0}{b_{1L}^0}$, then \eqref{main-eq1} has spatially homogeneous and time $T$-periodic coexistence state $(u_0^{**}(t),v_0^{**}(t))$, i.e., $\inf_{t\in\RR}u_0^{**}(t)>0$ and $\inf_{t\in\RR}v_0^{**}(t)>0$.
}

\item[$\bullet$] (Extinction) {\it If $a_{1L}^0> \frac{c_{1M}^0a_{2M}^0}{c_{2L}^0}$ and $a_{2M}^0< \frac{a_{1L}^0b_{2L}^0}{b_{1M}^0}$,
  then the semitrivial equilibrium $(u^*_0(t),0)$
of \eqref{main-eq1} is
globally stable.
}

\item[$\bullet$] (Extinction) {\it   If $a_{1M}^0< \frac{c_{1L}^0a_{2L}^0}{c_{2M}^0}$  and  $a_{2L}^0> \frac{a_{1M}^0b_{2M}^0}{b_{1L}^0}$, then the semitrivial equilibrium $(0,v^*_0(t))$ of \eqref{main-eq1}
is globally stable.
}
\end{itemize}

In \cite{KoRaSh}, the spreading speeds of \eqref{main-eq1} are investigated and the following result is proved.

\begin{itemize}
\item[$\bullet$] (Spreading speed for \eqref{main-eq1}) {\it Assume  $a_{1L}^0> \frac{c_{1M}^0a_{2M}^0}{c_{2L}^0}$ and $a_{2M}^0< \frac{a_{1L}^0b_{2L}^0}{b_{1M}^0}$.  Then \eqref{main-eq1} has  a single  spreading or invasion speed $c_0^*$ from $(u_0^*,0)$ to $(0,v_0^*)$ (see section \ref{spreading-speeds-section} for the definition of spreading speeds).
    }
\end{itemize}

As mentioned in the above, the objective of the current paper is to study the effect of localized spatial variations of the coefficients of \eqref{main-eq1}
on its asymptotic dynamics, in particular, on its persistence and spreading speeds.  Along this direction, first we have  following results  from  \cite[Theorem 2.1(3),(4)]{KoSh1}.

\begin{itemize}
\item[$\bullet$] (Semi-trivial solutions of \eqref{main-eq2}) {\it  There are two time periodic semi-trivial solutions $(u^*(t,x),0)$ and $(0,v^*(t,x))$ of \eqref{main-eq2} with
$\lim_{|x|\to \infty} [u^*(t,x)-u_0^*(t)]=0$ and $\lim_{|x|\to\infty}[v^*(t,x)-v^*_0(t)]=0$} (see Proposition \ref{semi-trivial-solu-prop}).
\end{itemize}

Let (H1)-(H2)  be the following standing  assumptions.

\medskip

\noindent {\bf (H1)} {\it $a_{1L}^0> \frac{c_{1M}^0a_{2M}^0}{c_{2L}^0}$ and $a_{2M}^0< \frac{a_{1L}^0b_{2L}^0}{b_{1M}^0}$}.

\smallskip

\noindent {\bf (H2)} {\it
$a^0_1(t)-c^0_1(t)\frac{a^0_{2M}}{c^0_{2L}}-a^0_2(t)+2c^0_2 (t)\frac{a^0_{2L}}{c^0_{2M}}
-b^0_2(t)\frac{a^0_{2M}}{c^0_{2L}}\frac{c^0_{1M}}{b^0_{1L}}> 0$  and
 $a^0_1(t)-c^0_1(t)\frac{a^0_{2M}}{c^0_{2L}}-a^0_2(t)+2c^0_2 (t)\frac{a^0_{2L}}{c^0_{2M}}-b^0_2(t)\frac{a^0_{2M}}{c^0_{2L}}\frac{c^0_{2M}}{b^0_{2L}}> 0$ for all $t\in\RR$.
}

\medskip

Note that (H1) implies that $(0,v_0^*)$ is an unstable solution of \eqref{main-eq1} and $(u_0^*,0)$ is a globally stable solution of
\eqref{main-eq1}.

Let $\lambda(\mu, a_1^0-c_1^0v_0^*)$ be the principal
spectrum point associated to the following linear equation,
\begin{equation}
\label{eigenvalue-eq0}
u_t=\mathcal{A}(\mu)u+(a_1^0(t)-c_1^0(t)v_0^*(t))u(t,x),
\end{equation}
where
\begin{equation}
\label{a-mu-eq1}
\mathcal{A}(\mu)u= u_{xx}+\mu^2 u
\end{equation}
when $\mathcal{A}=u_{xx}$, and
\begin{equation}
\label{a-mu-eq2}
\mathcal{A}(\mu)=\int_{\RR}e^{-\mu(y-x)}\kappa(y-x)u(t,y)dy-u(t,x)
\end{equation}
when $\mathcal{A}u=\int_{\RR}\kappa(y-x)u(t,y)dy-u(t,x)$
(see section 2 for the definition of principal spectrum point).

As recalled in the above.  Assume (H1). Then \eqref{main-eq1} has  a single  spreading or invasion speed $c_0^*$ from $(u_0^*,0)$ to $(0,v_0^*)$.
    If, in addition, (H2) holds, then we also have (see \cite{KoRaSh})
    $$
    c_0^*=\inf_{\mu>0}\frac{\lambda(\mu, a_1^0-c_1^0v_0^*)}{\mu}.
    $$
    Note that the spreading speed $c_0^*$ of \eqref{main-eq1} depends on $(a_1^0,b_1^0,c_1^0,a_2^0,b_2^0,c_2^0)$.
 To indicate the dependence of $c_0^*$ on $(a_1^0,b_1^0,c_1^0,a_2^0,b_2^0,c_2^0)$, we may write it as $c_0^*(a_1^0,b_1^0,c_1^0,a_2^0,b_2^0,c_2^0)$. It is unknown whether  $c_0^*(a_1^0,b_1^0,c_1^0,a_2^0,b_2^0,c_2^0)$
 depends on $(a_1^0,b_1^0,c_1^0,a_2^0,b_2^0,c_2^0)$ continuously.

\medskip

 In this paper, among others, we prove

 \begin{itemize}
 \item[$\bullet$]  (Stability/instability of semi-trivial solutions of \eqref{main-eq2})
{\it If $(u_0^*(t),0)$ (resp. $(0,v_0^*(t))$) is unstable, then so is $(u^*(t,x),0)$ (resp. $(0,v^*(t,x))$)}  (see Theorem \ref{stability-semitrivial-solu-thm}). (Note, if $(u_0^*(t),0)$ (resp. $(0,v_0^*(t))$) is stable, $(u^*(t,x),0)$ (resp. $(0,v^*(t,x))$)  may not be stable, see Remark \ref{stable-rk}).

\item[$\bullet$] (Persistence and coexistence of \eqref{main-eq2}) {\it If both $(u_0^*(t),0)$ and $(0,v_0^*(t))$ are unstable, then persistence occurs in \eqref{main-eq2} and \eqref{main-eq2} has a time $T$-periodic coexistence state $(u^{**}(t,x),v^{**}(t,x))$
} (see Theorem \ref{peristence-thm}).

\item[$\bullet$] (Lower semi-continuity of the spreading speeds of \eqref{main-eq1})
   {\it Assume (H1) and $(a_1^n(t),b_1^n(t)$, $c_1^n(t),a_2^n(t),b_2^n(t),c_2^n(t))$ are continuous $T$-periodic functions and converge to $(a_1^0(t),b_1^0(t)$, $c_1^0(t),a_2^0(t),b_2^0(t),c_2^0(t))$ as $n\to\infty$ uniformly in $t\in\RR$. Then
     $$
     \liminf_{n\to\infty} c_0^*(a_1^n,b_1^n,c_1^n,a_2^n,b_2^n,c_2^n)\ge c_0^*(a_1^0,b_1^0,c_1^0,a_2^0,b_2^0,c_2^0).
     $$
     If, in addition, (H2) holds, then $c_0^*$ is continuous at $(a_1^0,b_1^0,c_1^0,a_2^0,b_2^0,c_2^0)$, that is,
     $$
     \lim_{n\to\infty} c_0^*(a_1^n,b_1^n,c_1^n,a_2^n,b_2^n,c_2^n)= c_0^*(a_1^0,b_1^0,c_1^0,a_2^0,b_2^0,c_2^0)
     $$
     }
     (see Theorem \ref{main-thm0}).

\item[$\bullet$] (Spreading speeds of \eqref{main-eq2}) {\it Assume (H1). Then
$$c_{\sup}^*\ge c_{\inf}^*\ge c_0^*,$$
where $[c_{\inf}^*,c_{\sup}^*]$ is the spreading speed interval of \eqref{main-eq2} (see section \ref{spreading-speeds-section} for the definition of spreading speed
interval).
If, in addition, assume (H2), then
$$c_{\sup}^*=c_{\inf}^*=c_0^*.$$
 } (see Theorem \ref{main-thm1}).
\end{itemize}

We remark that (H2) is a linear determinant condition for \eqref{main-eq1}.  If $a_i^0(t)$, $b_i^0(t)$, and $c_i^0(t)$ ($i=1,,2$) are independent of $t$, then (H2) becomes
\begin{equation*}
a_1^0+a_2^0-\frac{a_2^0c_1^0}{c_2^0}-\frac{a_2^0b_2^0c_1^0}{b_1^0c_2^0}> 0,\quad
a_1^0-\frac{a_2^0c_1^0}{c_2^0}> 0.
\end{equation*}
In particular, if
$$
a_1^0=r_1,\quad b_1^0=r_1,\quad c_1^0=\tilde a_1^0 r_1
$$
and
$$
a_2^0=r_2,\quad b_2^0=r_2\tilde a_2^0,\quad c_2^0=r_2
$$
with
$$\tilde a_1^0<1\le \tilde a_2^0,
$$
then
(H2) becomes
\begin{equation}
\label{HL2-eq2}
\frac{\tilde a_1\tilde a_2 -1}{1-\tilde a_1}\le \frac{r_1}{r_2},
\end{equation}
which is
 the same as the linear determinant condition for \eqref{main-eq1} in \cite[Theorem 2.1]{LeLiWe}.

The results stated in the above reveal several interesting scenarios, for example, localized spatial perturbation does not affect the instability of the
semitrivial solutions of \eqref{main-eq1};  localized spatial variation does not slow down the spreading speeds or invasion speeds of the stable species to the unstable species; and if the linear determinant condition (H2) holds,  then $c_0^*(a_1^0,b_1^0,c_1^0,a_2^0,b_2^0,c_2^0)$ is continuous with respect to the coefficients $(a_1^0,b_1^0,c_1^0,a_2^0,b_2^0,c_2^0)$ and any localized variation does not speed up the spreading speeds.
It remains open whether in general  $c_0^*(a_1^0,b_1^0,c_1^0,a_2^0,b_2^0,c_2^0)$ is continuous with respect to the coefficients $(a_1^0,b_1^0,c_1^0,a_2^0,b_2^0,c_2^0)$, and  whether in general localized spatial variation does not speed up the spreading speeds.

The rest of the paper is organized as follows. In section 2, we present some preliminary materials to be used in the proof of the main results.
 Section \ref{semitrivial-solution-section} is devoted to the study of semitrivial solutions of \eqref{main-eq2}.
We study the persistence, coexistence, and extinction
of \eqref{main-eq2} in section \ref{persistence-section}. In section \ref{spreading-speeds-section}, we investigate the spreading speeds of \eqref{main-eq2}.

\section{Preliminary}

In this section, we present some preliminary materials to be used in the proof of the main results,  including  some principal spectrum theory
for dispersal operators with time periodic dependence and comparison principle for competitive/cooperative systems.

\subsection{Principal spectrum theory}

In this subsection, we present some principal spectrum theory for dispersal operators with time periodic dependence
 to be used in later sections.

Let $a(t,x)$ be H\"older continuous in $t$ and continuous in $x$, $a(t+T,x)=a(t,x)$, and $\mu\ge 0$ be a given constant. Consider
\begin{equation}
\label{main-linear-eq1}
u_t=\mathcal{A}(\mu) u+a(t,x)u,\quad x\in\RR,
\end{equation}
where $\mathcal{A}(\mu)u$ is as in \eqref{a-mu-eq1} if $\mathcal{A}u=u_{xx}$ and $\mathcal{A}(\mu)u$ is as in \eqref{a-mu-eq2}
if $\mathcal{A}u=\int_{\RR}\kappa(y-x)u(t,y)dy-u(t,x)$.

Let
\begin{equation}
\label{space-x-eq1}
X=C_{\rm unif}^b(\RR)
\end{equation}
with norm $\|u\|=\sup_{x\in\RR}|u(x)|$,
\begin{equation}
\label{space-x-eq2}
X^+=\{u\in X\,|\, u(x)\ge 0\quad \forall\,\, x\in\RR\},
\end{equation}
and
\begin{equation}
\label{space-x-eq3}
X^{++}=\{u\in X^+\,|\, \inf_{x\in\RR} u(x)>0\}.
\end{equation}
For given $u_1,u_2\in X$, we define
$$
u_1\le u_2\quad {\rm if}\quad u_2-u_1\in X^+.
$$
Let $\Phi(t,s;\mu, a)$ be the solution operator of \eqref{main-linear-eq1} on $X$, that is, for any $u_0\in X$,
$u(t,\cdot;s,u_0):=\Phi(t,s;\mu, a)u_0$ is the solution of \eqref{main-linear-eq1} with $u(s,\cdot;s,u_0)=u_0(\cdot)$.

Observe that for any $u_0\in X^+$, $\Phi(t,s;\mu, a)u_0\in X^+$ for all $t\ge s$. Hence, for any $u_1,u_2\in X$ with
$u_1\le u_2$,
$$
\Phi(t,s;\mu,a)u_1\le \Phi(t,s;\mu,a)u_2\quad \forall \,\, t\ge s.
$$
Observe also that, if $a_1(t,x)\le a_2(t,x)$, then for any $u_0\in X^+$,
$$
\Phi(t,s;\mu,a_1)u_0\le\Phi(t,s;\mu,a_2)u_0\quad \forall\,\, t\ge s.
$$

\begin{definition}
\label{principal-spectrum-def}
$\lambda(\mu, a)=\limsup_{t-s\to\infty}\frac{\ln \|\Phi(t,s;\mu,a)\|}{t-s}$ is called the {\it principal spectrum point} or {\it generalized principal
eigenvalue} of \eqref{main-linear-eq1}.
\end{definition}

If no confusion occurs, we may write $\lambda(0,a)$ and $\Phi(t,s;0,a)$ as $\lambda(a)$ and $\Phi(t,s;a)$, respectively.

\begin{proposition}
\label{principal-spectrum-prop1} Let $a_0(t)$ be   H\"older continuous in $t$, $a_0(t+T)=a_0(t)$, and  $a(t,x)$ be H\"older continuous in $t$ and continuous in $x$, $a(t+T,x)=a(t,x)$.
\begin{itemize}
\item[(1)] $\lambda(\mu,a)=\frac{\ln r(\Phi(T,0;\mu,a))}{T}$, where $r(\Phi(T,0;\mu,a))$ is the spectral radius of $\Phi(T,0;\mu,a)$ on $C_{\rm unif}^b(\RR)$.

\item[(2)]   $\lambda(a_0)=\hat a_0:=\frac{1}{T}\int_0^T a_0(t)dt$.

\item[(3)] If $\lim_{|x|\to\infty}|a(t,x)-a_0(t)|=0$, then $\lambda(a)\ge \lambda(a_0)$.

\item[(4)] If $a(t,x)$ is also periodic in $x$ with period $p$ and is $C^1$, then there is a positive function $\phi(t,x)$, $\phi(t+T,x)=\phi(t,x+p)=\phi(t,x)$, such that
   $ \Phi(T,0;\mu,a)\phi(0,\cdot)=e^{\lambda(\mu,a)T}\phi(T,\cdot)$.
\end{itemize}
\end{proposition}

\begin{proof}
(1) It can be proved by the similar arguments  as those in  \cite[Proposition 3.3]{RaSh}.

(2) Let $u_0\equiv 1$. Then
$$
\Phi(t,s;a_0)u_0=e^{\int_s^t a_0(\tau)d\tau} u_0\quad \forall\,\, t\ge s.
$$
This implies that
$$
\|\Phi(t,s;a_0)\|=e^{\int_s^ t a_0(\tau)d\tau} \quad \forall\,\, t\ge s.
$$
Hence
$$
\lambda(a_0)=\limsup_{t-s\to\infty}\frac{\ln\|\Phi(t,s;a_0)\|}{t-s}=\hat a_0.
$$

(3) For any $\epsilon>0$, by the arguments in \cite[Lemma 4.1]{KoSh}, there  is a space and time periodic function $\tilde a(t,x)$, $\tilde a(t+T,x)=\tilde a(t,x+L)=\tilde a(t,x)$,
 such that
$$
a(t,x)\ge \tilde a(t,x)\quad \forall \,\, t,x\in\RR
$$
and
$$
\bar {\tilde a}(t):=\frac{1}{L}\int_0^L \tilde a(t,x)dx\ge a_0(t)-\epsilon.
$$
By \cite[Propositions 3.3, 3.10, and Theorem C]{RaSh}, we have
$$
\lambda(a)\ge \lambda(\tilde a)\ge \lambda(\hat {\tilde a}),
$$
where
$$
\hat{\tilde a}(x)=\frac{1}{T}\int_0^T \tilde a(t,x)dt.
$$
By \cite[Theorem 2.1]{HeShZh}, we have
$$
\lambda(a)\ge \hat a_0-\epsilon= \lambda(a_0)-\epsilon.
$$
Letting $\epsilon\to 0$, we have $\lambda(a)\ge \lambda(a_0)$.

(4) It follows from \cite[Theorem B]{RaSh}.
\end{proof}

Consider the following nonhomogeneous linear equation,
\begin{equation}
\label{nonhomogeneous-eq} \frac{\p u}{\p t}=\mathcal{A}(\mu)u(t,x)+a(t)u(t,x)+h(t),\quad x\in\RR,
\end{equation}
where $a(t)$ and $h(t)$ are  $T$-periodic continuous functions. We have

\begin{proposition}
\label{nonhomogeneous-prop}
If $\lambda(\mu,a)<0$, then \eqref{nonhomogeneous-eq}
has a unique $T$-periodic solution $u^{**}(t)$. Moreover, $u^{**}(t)$ is  a globally stable solution
of \eqref{nonhomogeneous-eq}  with respect to perturbations in $X_p$, and if $h(t)\ge 0$ and $h(t)\not \equiv 0$, then
$\inf_{t\in\RR} u^{**}(t)>0$.
\end{proposition}

\begin{proof}
It follows from \cite[Proposition 2.5]{KoRaSh}
\end{proof}

\subsection{Comparison principle}

In this subsection, we recall some comparison principle for competitive and cooperative systems.

Let $X$, $X^+$, and $X^{++}$ be as in \eqref{space-x-eq1}, \eqref{space-x-eq2},
and \eqref{space-x-eq3}, respectively.
For given $(u_0,v_0)\in X\times X$, let $(u(t,x;t_0,u_0,v_0),v(t,x;t_0,u_0,v_0))$ be the (local) solution of
\eqref{main-eq2} with $(u(t_0,x;t_0,u_0,v_0)$, $v(t_0,x;t_0,u_0,v_0))=(u_0(x),v_0(x))$. We put
$$(u(t,x;u_0,v_0), v(t,x;u_0,v_0))=(u(t,x;0,u_0,v_0),v(t,x;0,u_0,v_0)).
$$
Note that for any $(u_0,v_0)\in X^+\times X^+$ and $t_0\in\RR$, $(u(t,x;t_0,u_0,v_0),v(t,x;t_0,u_0,v_0))$
exists for all $t\ge t_0$ and $(u(t,\cdot;t_0,u_0,v_0),v(t,\cdot;t_0,u_0,v_0))\in X^+\times X^+$ for all $t\ge t_0$.
For biological reason, we are only interested in nonnegative solutions of \eqref{main-eq2}.

We call $(u(t,x),v(t,x))$ is a super-solution (sub-solution) of \eqref{main-eq2} for $t$ in an interval $I$  if
$u(t,\cdot)$, $v(t,\cdot)\in X$ for $t\in I$ and
$$
\begin{cases}
u_t\ge (\le) \mathcal{A}u+u(a_1(t,x)-b_1(t,x)u-c_1(t,x)v),\quad t\in {\rm Int}(I),\,\, x\in\RR\cr
v_t\le (\ge)\mathcal{A} v+v(a_2(t,x)-b_2(t,x)u-c_2(t,x)v),\quad t\in {\rm Int}(I),\,\, x\in\RR.
\end{cases}
$$

\begin{proposition}
\label{comparison-prop}
Assume that $(u_1(t,x),v_1(t,x))$ and $(u_2(t,x),v_2(t,x))$ are sub-solution and super-solution of \eqref{main-eq2} for $t\in [t_1,t_2)$
and $0\le u_1(t_1,x)\le u_2(t_1,x)$, $v_1(t_1,x)\ge v_2(t_1,x)\ge 0$ for $x\in\RR$. Then
$$
0\le u_1(t,x)\le u_2(t,x),\quad v_1(t,x)\ge v_2(t,x)\quad \forall\,\, t\in (t_1,t_2),\quad x\in\RR.
$$
\end{proposition}

\begin{proof}
It follows from comparison principle for two species competitive systems of parabolic equations for the case that $\mathcal{A}u=u_{xx}$ and  follows from the arguments in \cite[Proposition 3.1]{HeNgSh} for the case that $(\mathcal{A}u)(t,x)=\int_{\RR}\kappa(y-x)u(t,y)dy-u(t,x)$.
\end{proof}

Consider
\begin{equation}
\label{cooperative-eq}
\begin{cases}
u_t=\mathcal{A}u+f(t,x,u,v),\quad x\in\RR\cr
v_t=\mathcal{A}v+g(t,x,u,v),\quad x\in\RR,
\end{cases}
\end{equation}
where $F(t,x,u, v)$ and $g(t,x,u,v)$ are locally H\"older continuous in $t$,  uniformly continuous in $x$, and $C^1$ in $u,v$, $f(t,x,0,0)=0$, $g(t,x,0,0)=0$, and
$f_v(t,x,u,v)\ge 0$, $g_u(t,x,u,v)\ge 0$ for $u\ge 0$, $v\ge 0$.
$(u(t,x),v(t,x))\ge (0,0)$ is called a super-solution (sub-solution) of \eqref{cooperative-eq}  on $(\xi^*(t),\infty)$ ($\xi^*(t)\ge -\infty$) for
$t\ge 0$ if $(u(t,x),v(t,x))$ is continuous in $t$ and $x$, and satisfies
$$
\begin{cases}
u_t\ge (\le) \mathcal{A}u+f(t,x,u,v),\quad x>\xi^*(t)\cr
v_t\ge (\le)  \mathcal{A}v+g(t,x,u,v),\quad x>\xi^*(t)
\end{cases}
$$
for $t\ge 0$.

\begin{proposition}
\label{comparison-cooperative-prop}
Suppose that $(u^+(t,x),v^+(t,x))\ge (0,0)$ is  a super-solution  and $(u^-(t,x),v^-(t,x))\ge (0,0)$ is a sub-solution of \eqref{cooperative-eq}  on $(\xi^*(t),\infty)$ ($\xi^*(t)\ge -\infty$) for
$t\ge 0$.

\begin{itemize}
\item[(1)] If $\xi^*(t)=-\infty$ for $t\ge 0$ and $u^+(0,x)\ge u^-(0,x)$, $v^+(0,x)\ge v^-(0,x)$ for $x\in\RR$, then
$$
u^+(t,x)\ge u^-(t,x),\quad v^+(t,x)\ge v^-(t,x)\quad \forall\,\, t>0, \,\, x\in\RR.
$$

\item[(2)] If $\xi^*(\cdot): [0,\infty)\to (-\infty,\infty)$ is $C^1$,  $u^+(t,x)\ge u^-(t,x)$ and $v^+(t,x)\ge v^-(t,x)$ for
$t\ge 0$, $x\le \xi^*(t)$, and $u^+(0,x)\ge u^-(0,x)$ and $v^+(0,x)\ge v^-(0,x)$ for $x\ge \xi^*(0)$, then
$$
u^+(t,x)\ge u^-(t,x),\quad v^+(t,x)\ge v^-(t,x),\quad \forall\,\, t\ge 0,\,\, x\ge \xi^*(t).
$$
\end{itemize}
\end{proposition}

\begin{proof}
It follows from comparison principle for cooperative systems of parabolic equations for the case that $\mathcal{A}u=u_{xx}$ and  follows from the arguments in \cite[Proposition 2.1]{ShZh1} for the case that $(\mathcal{A}u)(t,x)=\int_{\RR}\kappa(y-x)u(t,y)dy-u(t,x)$.
\end{proof}

\section{Semi-trivial solutions}\label{semitrivial-solution-section}

Consider
\begin{equation}
\label{one-species-eq}
w_t=\mathcal{A}w+w(a(t,x)-b(t,x)w),\quad x\in\RR,
\end{equation}
where $a(t+T,x)=a(t,x)$ and $b(t+T,x)=b(t,x)$;  $a(t,x)-a_0(t)\to 0$ and
$b(t,x)-b_0(t)\to 0$ as $|x|\to\infty$ uniformly in $t$; and $a(t,x)$, $b(t,x)$, $a_0(t)$, and $b_0(t)$ are uniformly continuous in $x\in\RR$ and locally H\"older continuous in $t$. Let $w(t,x;t_0,w_0)$ be the solution of \eqref{one-species-eq} with $w(t_0,\cdot;t_0,w_0)=w_0\in X$.

\begin{proposition}
\label{one-species-prop}
Assume that  $\inf_{t,x\in\RR}b(t,x)>0$. If $\lambda(a_0)>0$, then there is a  unique time periodic
 positive solution $w^*(t,x;a,b)$ of \eqref{one-species-eq} satisfying that $\inf_{t,x\in\RR} u^*(t,x;a,b)>0$, and that  for any  $w_0\in X^{++}$,
$$
\lim_{t\to\infty} |w(t+t_0,x;t_0,w_0)-w^*(t+t_0,x;a,b)|=0
$$
uniformly in $x\in\RR$ and $t_0\in\RR$, and
$$
\lim_{|x|\to\infty} w^*(t,x)=w_0^*(t;a_0,b_0)
$$
uniformly in $t\in\RR$, where $w_0^*(t;a_0,b_0)$ is the unique time periodic positive solution of
$$
w_t=w(a_0(t)-b_0(t)w).
$$
\end{proposition}

\begin{proof}
It follows from \cite[Theorem 2.1(2)-(4)]{KoSh1}.
\end{proof}

\begin{proposition}
\label{semi-trivial-solu-prop}
\begin{itemize}
\item[(1)] There is a unique semitrivial periodic solution $(u^*(t,x),0)$ of \eqref{main-eq2} satisfying that
$\inf_{t,x\in\RR} u^*(t,x)>0$, and that
$$
\lim_{|x|\to\infty} |u^*(t,x)-u_0^*(t)|=0
$$
uniformly in $t\in\RR$, and for any $u_0\in X^{++}$,
$$
\lim_{t\to\infty} |u(t,x;u_0,0)-u^*(t,x)|+|v(t,x;u_0,0)|=0
$$
uniformly in $x\in\RR$.

\item[(2)] There is a unique semitrivial periodic solution $(0,v^*(t,x))$ of \eqref{main-eq2} satisfying that
$\inf_{t,x\in\RR}v^*(t,x)>0$, and that
$$
\lim_{|x|\to\infty} |v^*(t,x)-v_0^*(t)|=0
$$
uniformly in $t\in\RR$, and  for any $v_0\in X^{++}$,
$$
\lim_{t\to\infty} |u(t,x;0,v_0)|+|v(t,x;0,v_0)-v^*(t,x)|=0
$$
uniformly in $x\in\RR$.
\end{itemize}
\end{proposition}

\begin{proof}
It follows from Proposition \ref{one-species-prop}.
\end{proof}

We call $(u^*(t,x),0)$ and $(0,v^*(t,x))$  {\it semitrivial solutions} of
\eqref{main-eq2}. Consider the linearization of \eqref{main-eq2} at $(u^*,0)$ and $(0,v^*)$,
\begin{equation}
\label{linear-eq1}
\begin{cases}
u_t=\mathcal{A} u+(a_1(t,x)-2b_1u^*(t,x))u-c_1(t,x)u^*(t,x)v,\quad x\in\RR\cr
v_t=\mathcal{A} v+(a_2(t,x)-b_2(t,x)u^*(t,x))v,\quad x\in\RR,
\end{cases}
\end{equation}
and
\begin{equation}
\label{linear-eq2}
\begin{cases}
u_t=\mathcal{A}+(a_1(t,x)-c_1(t,x)v^*(t,x))u,\quad x\in\RR\cr
v_t=\mathcal{A} v-b_2(t,x)v^*(t,x)u+(a_2(t,x)-2 c_2(t,x)v^*(t,x))v,\quad x\in\RR.
\end{cases}
\end{equation}
Let $\Psi_1(t,s;u^*,0)$ be the solution operator of \eqref{linear-eq1} and $\Psi_2(t,s;0,v^*)$ be the solution operator of
\eqref{linear-eq2} on $X$.
We call $(u^*,0)$ {\it linearly unstable} if  $r(\Psi_1(T,0;u^*,0))>1$ and call $(0,v^*)$ {\it linearly unstable} if $r(\Psi_2(T,0;0,v^*))>1$.

\begin{theorem}
\label{stability-semitrivial-solu-thm}
\begin{itemize}
\item[(1)] If $r(\Phi(T,0;a_2-b_2u^*))>1$, then $(u^*,0)$ is linearly unstable. If $r(\Phi(T,0;a_1-c_1v^*))>1$,
then $(0,v^*)$ is linearly unstable.

\item[(2)] If $r(\Phi(T,0;a_2^0-b_2^0 u_0^*)>1$, then
$r(\Psi_1(T,0;u^*,0))>1$. Hence if $(u^*_0,0)$ is  a linearly unstable solution of \eqref{main-eq1}, then
$(u^*,0)$ is a linearly  unstable solution of \eqref{main-eq2}.

\item[(3)] If $r(\Phi(T,0;a_1^0-c_1^0 v_0^*)>1$, then
$r(\Psi_2(T,0;0,v^*))>1$. Hence if $(0,v^*_0)$ is  a linearly  unstable solution of \eqref{main-eq1}, then
$(0,v^*)$ is a linearly  unstable solution of \eqref{main-eq2}.
\end{itemize}
\end{theorem}

\begin{proof}
(1) For given $(u_0,v_0)\in X\times X$, let
$(\Psi_{11}(T,0;u^*,0)(u_0,v_0),\Psi_{12}(T,0;u^*,0)(u_0,v_0))=\Psi_1(T,0;u^*,0)(u_0,v_0)$. Then
$\Psi_{12}(T,0;u^*,0)(u_0,v_0)=\Phi(T,0;a_2-b_2u^*))v_0$. Hence
$$
r(\Psi_1(T,0;u^*,0))\ge r(\Phi(T,0;a_2-b_2u^*)).
$$
This implies that if $r(\Phi(T,0;a_2-b_2u^*))>1$, then $(u^*,0)$ is linearly unstable.

Similarly, we can prove that, if $r(\Phi(T,0;a_1-c_1v^*))>1$,
then $(0,v^*)$ is linearly unstable.

(2) By Proposition \ref{semi-trivial-solu-prop}(1),
$$
\lim_{|x|\to\infty} |\big(a_2(t,x)-b_2(t,x)u^*(t,x)\big)-\big(a_2^0(t)-b_2^0(t)u_0^*(t)\big)|=0
$$
uniformly in $t\in\RR$. Then by Proposition \ref{principal-spectrum-prop1} and (1),
$$
r(\Psi_1(T,0;u^*,0))\ge r(\Phi(T,0;a_2-b_2u^*))\ge r(\Phi(T,0;a_2^0-b_2^0 u_0^*)).
$$
This implies (2).

(3) It can be proved by the similar arguments as in (2).
\end{proof}

\begin{remark}
\label{stable-rk}
If $(u_0^*,0)$ is a stable solution of \eqref{main-eq1}, $(u^*,0)$ may not be a stable solution of \eqref{main-eq2}.
Similarly, if $(0,v_0^*)$ is a stable solution of \eqref{main-eq1}, $(0,v^*)$ may not be a stable solution of \eqref{main-eq2}.
In fact, by the arguments in Theorem \ref{stability-semitrivial-solu-thm}(1), (2),
$$
r(\Psi_1(T,0;u^*,0))\ge r(\Phi(T,0;a_2-b_2u^*))\ge  r(\Phi(T,0;a_2^0-b_2^0 u_0^*).
$$
Assume that $\lambda(a_2^0-b_2^0u_0^*)<0$. Let  $a_1(t,x)=a_1^0(t)$, $b_1(t,x)=b_1^0(t)$, $c_1(t,x)=c_1^0(t)$, $b_2(t,x)=b_2^0(t)$, $c_2(t,x)=c_2^0(t)$, and
$$
a_2(t,x)=a_2^0(t)+a^*(x),
$$
where $a^*\in X^+$ with compact support and $\lambda(a^*)>-\frac{1}{T}\int_0^T \big(a_2^0(t)-b_2^0(t)u_0^*(t)\big)dt$. Then
$$
a_2(t,x)-b_2(t,x)u^*(t,x)=a^*(x)+a_2^0(t)-b_2^0(t)u_0^*(t)
$$
and
$$
\lambda(a_2-b_2 u^*)=\lambda (a^*)+\frac{1}{T}\int_0^T \big(a_2^0(t)-b_2^0(t)u_0^*(t)\big)dt>0.
$$
So in this case, $(u_0^*,0)$ is linearly stable solution of \eqref{main-eq1} and $(u^*,0)$ is linearly unstable
solution of \eqref{main-eq2}.

Similarly, if $(0,v_0^*)$ is a stable solution of \eqref{main-eq1}, $(0,v^*)$ may not  be  a stable solution of \eqref{main-eq2}.
\end{remark}

\section{Persistence and Coexistence}\label{persistence-section}

In this section, we study the persistence and coexistence dynamics of \eqref{main-eq2}.

We say that {\it persistence occurs} in \eqref{main-eq2} if there is  $\eta>0$ such that for any
 $(u_0,v_0)\in X^{++}\times X^{++}$, there is $T(u_0,v_0)>0$ such that
$$
 u(t+t_0,x;t_0,u_0,v_0)\ge \eta,\,\,\, v(t+t_0,x;t_0,u_0,v_0)\ge \eta\quad \forall \,\, t\ge T(u_0,v_0),\,\, x\in\RR,\,\, t_0\in\RR.
 $$
A time $T$-periodic  solution $(u^{**}(t,x),v^{**}(t,x))$ of \eqref{main-eq2} is called a {\it coexistence state} if
$$\inf_{t\in\RR,x\in\RR}u^{**}(t,x)>0\quad {\rm and}\quad \inf_{t\in\RR,x\in\RR}v^{**}(t,x)>0.
$$

\begin{theorem}
\label{peristence-thm}
 If both $(u_0^*,0)$ and $(0,v_0^*)$ are linearly unstable solutions of \eqref{main-eq1}, then
persistence occurs in \eqref{main-eq2} and \eqref{main-eq2} has a coexistence state $(u^{**}(t,x),v^{**}(t,x))$.
If, in addition, $\frac{\inf_{t\in\RR}b_{1}(t,x)}{\sup_{t\in\RR}b_{2}(t,x)}>\frac{\sup_{t\in\RR}c_{1}(t,x)}{\inf_{t\in\RR}c_{2}(t,x)}$ for each $x\in\RR$, then \eqref{main-eq2} has a spatially continuous coexistence state $(u^{**}(t,x),v^{**}(t,x))$.
\end{theorem}

\begin{proof}
Assume that $(u_0^*,0)$ and $(0,v_0^*)$ are linearly unstable solutions of \eqref{main-eq1}. We first prove that the persistence occurs in \eqref{main-eq2}.

By the linear instability of $(0,v_0^*)$,
$\lambda(a_1^0-c_1^0v^*_0)>0$. Hence, by Proposition \ref{principal-spectrum-prop1}, there is $\epsilon_0>0$ such that for any $0<\epsilon\le \epsilon_0$,
\begin{equation}
\label{persistence-eq1}
 \lambda(a_1-c_1(v^*+\epsilon))\ge \lambda(a_1^0-c_1^0(v^*_0+\epsilon)) >0.
\end{equation}
For any
 $(u_0,v_0)\in X^{++}\times X^{++}$, by Propositions \ref{comparison-prop}  and  \ref{semi-trivial-solu-prop} and Theorem
 \ref{stability-semitrivial-solu-thm}, there is $n_0\in\NN$ such that
 $$ 0< v(t_0+t,x;t_0,u_0,v_0)\le v(t_0+t,x;t_0,0,v_0)\le v^*(t_0+t,x)+\epsilon_0/2
 $$
 for $t\ge n_0T$,  $x\in\RR$, and $t_0\in\RR$.
 This implies that
 $$
 u_t\ge \mathcal{A}u+u(a_1(t,x)-b_1(t,x)u-c_1(t,x)(v^*(t,x)+\epsilon_0/2))
$$
for $t\ge t_0+n_0T$. Let $0<\epsilon_1<\epsilon_0$ be such that
$$
\inf_{t\in\RR,x\in\RR}\big(w^*(t,x;a_1-c_1(v^*+\epsilon_0),b_1)-\epsilon_1\big)>0,
$$
where $w^*(t,x;a_1-c_1(v^*+\epsilon_0),b_1)$ is the unique $T$-periodic positive solution of \eqref{one-species-eq} with $a$ being replaced
by $a_1-c_1(v^*+\epsilon_0)$ and $b$ being replaced by $b_1$.
Then by Proposition \ref{one-species-prop} and \eqref{persistence-eq1}, there is $n_1\ge n_0$ such that
$$
w^*(t_0+t,x;a_1-c_1(v^*+\epsilon_0),b_1)-\epsilon_1\le u(t_0+t,x;t_0,u_0,v_0)\le u^*(t_0+t,x)+\epsilon_0/2
$$
for $t\ge n_1T$, $x\in\RR$, and $t_0\in\RR$.

Let $\eta_1= \inf_{t\in\RR,x\in\RR} \big(w^*(t,x;a_1-c_1(v^*+\epsilon_0),b_1)-\epsilon_1\big)$.
We then have that for any  $(u_0,v_0)\in X^{++}\times X^{++}$, there is $T_1(u_0,v_0)>0$ such that
$$
u(t_0+t,x;t_0,u_0,v_0)\ge \eta_1
$$
for all $t\ge T_1(u_0,v_0)$, $x\in\RR$, and $t_0\in\RR$.

Similarly, by the linear instability of $(u_0^*,0)$, we can prove that there is $\eta_2>0$ such that
for any  $(u_0,v_0)\in X^{++}\times X^{++}$, there is $T_2(u_0,v_0)>0$ such that
$$
v(t_0+t,x;t_0,u_0,v_0)\ge \eta_2
$$
for all $t\ge T_2(u_0,v_0)$, $x\in\RR$, and $t_0\in\RR$. Let $\eta=\min\{\eta_1,\eta_2\}$. Then for any $(u_0,v_0)\in X^{++}\times X^{++}$, there is $T(u_0,v_0)>0$ such that
$$
u(t_0+t,x;t_0,u_0,v_0)\ge \eta,\quad v(t_0+t,x;t_0,u_0,v_0)\ge \eta
$$
for all $t\ge T(u_0,v_0)$, $x\in\RR$, and $t_0\in\RR$. Hence uniform persistence occurs in  \eqref{main-eq2}.

Next, we prove the existence of a coexistence state.
By the arguments in \cite[Lemma 4.1]{KoSh}, there are $\delta>0$, $L>0$, and $C^1$ functions
$a^*(t,x)$ and $b^*(t,x)$ satisfying that
$$
a^*(t+T,x)=a^*(t,x+L)=a^*(t,x),\quad b^*(t+T,x)=b^*(t,x+L)=b^*(t,x),
$$
$$
a_2(t,x)-b_2(t,x)u^*(t,x)\ge a^*(t,x)+\delta,\quad a_1(t,x)-c_1(t,x)v^*(t,x)\ge b^*(t,x)+\delta,
$$
and
$$
\lambda(a^*)>0,\quad  \lambda(b^*)>0.
$$
Moreover, there are positive functions $\phi^*(t,x)$ and $\psi^*(t,x)$ satisfying that
$$
\Phi(T,0;a^*)\phi^*(0,\cdot)=e^{\lambda(a^*)T}\phi^*(0,\cdot),\quad \Phi(T,0;b^*)\psi^*(0,\cdot)=e^{\lambda(b^*)T}\psi^*(0,\cdot).
$$
It is then not difficult to prove that there is $\tilde \epsilon_0>0$ such that for any $0<\epsilon\le \tilde \epsilon_0$, $(u_\epsilon^+(t,x),v_\epsilon^+(t,x))=(u^*(t,x),\epsilon e^{\lambda(a^*)t}\phi^*(t,x))$
is super-solution of \eqref{main-eq2} for $0\le t\le T$ and
$(u_\epsilon^-(t,x),v_\epsilon^-(t,x))=(\epsilon e^{\lambda(b^*)t} \psi^*(t,x),v^*(t,x))$ is sub-solution of \eqref{main-eq2} for $0\le t\le T$.

Fix $0<\epsilon\le \tilde \epsilon_0$ such that
$$
u_\epsilon^-(0,x)<u_\epsilon^+(0,x),\quad v_\epsilon^-(0,x)>v_\epsilon^+(0,x)\quad \forall\,\, x\in\RR.
$$
By Proposition \ref{comparison-prop},
$$
u((n+1)T+t,x;0,u_\epsilon^+(0,\cdot),v_\epsilon^+(0,\cdot))\le u(nT+t,x;0,u_\epsilon^+(0,\cdot),v_\epsilon^+(0,\cdot)),
$$
$$
v((n+1)T+t,x;0,u_\epsilon^+(0,\cdot),v_\epsilon^+(0,\cdot))\ge v(nT+t,x;0,u_\epsilon^+(0,\cdot),v_\epsilon^+(0,\cdot)),
$$
$$
u((n+1)T+t,x;0,u_\epsilon^-(0,\cdot),v_\epsilon^-(0,\cdot))\ge u(nT+t,x;0,u_\epsilon^-(0,\cdot),v_\epsilon^-(0,\cdot)),
$$
and
$$
v((n+1)T+t,x;0,u_\epsilon^-(0,\cdot),v_\epsilon^-(0,\cdot))\le v(nT+t,x;0,u_\epsilon^-(0,\cdot),v_\epsilon^-(0,\cdot))
$$
for $n=0,1,2,\cdots$, $t\ge 0$,  and $x\in\RR$. Let
$$
(u^+(t,x),v^+(t,x))=\lim_{n\to\infty} (u(nT+t,x;0,u_\epsilon^+,v_\epsilon^+), v(nT+t,x;0,u_\epsilon^+,v_\epsilon^+))
$$
and
$$
(u^-(t,x),v^-(t,x))=\lim_{n\to\infty} (u(nT+t,x;0,u_\epsilon^-,v_\epsilon^-), v(nT+t,x;0,u_\epsilon^-,v_\epsilon^-))
$$
for $t\ge 0$ and $x\in\RR$. Then
$$
(u^\pm(t+T,x),v^\pm(t+T,x))=(u^\pm(t,x),v^\pm(t,x))\quad \forall\,\, t\ge 0,\quad x\in\RR.
$$

In the case that $\mathcal{A}u=u_{xx}$, by the regularity and a priori estimates for parabolic equations,
both $(u^+(t,x),v^+(t,x))$ and $(u^-(t,x),v^-(t,x))$ are continuous coexistence states of \eqref{main-eq2}.
In the case that $(\mathcal{A}u)(t,x)=\int_{\RR}k(y-x)u(t,y)dy-u(t,x)$, by the arguments of \cite[Theorem A]{NgRa},
both $(u^+(t,x),v^+(t,x))$ and $(u^-(t,x),v^-(t,x))$ are semi-continuous coexistence states of \eqref{main-eq2}, and moreover, if
$\frac{\inf_{t\in\RR}b_{1}(t,x)}{\sup_{t\in\RR}b_{2}(t,x)}>\frac{\sup_{t\in\RR}c_{1}(t,x)}{\inf_{t\in\RR}c_{2}(t,x)}$ for each $x\in\RR$, then
they are continuous coexistence states of \eqref{main-eq2}.
This completes the proof of the theorem.
\end{proof}

\begin{corollary}
\label{persistence-cor}
 If $(0,v_0^*)$ is a  linearly unstable solution of \eqref{main-eq1}, then there is $\eta>0$ such that for any
 $(u_0,v_0)\in X^{++}\times X^{++}$, there is $T(u_0,v_0)>0$ such that
$$
 u(t+t_0,x;t_0,u_0,v_0)\ge \eta,\,\,\, v(t+t_0,x;t_0,u_0,v_0)\le v^*(t+t_0,x)- \eta\quad \forall \,\, t\ge T(u_0,v_0),\,\, x\in\RR,\,\, t_0\in\RR.
 $$
\end{corollary}

\begin{proof}
It follows from the arguments of Theorem \ref{peristence-thm}. To be more precise,
assume that  $(0,v_0^*)$ is a linearly unstable solution of \eqref{main-eq1}.
By the arguments of Theorem \ref{peristence-thm}, there is $\eta_1>0$ such that
for any  $(u_0,v_0)\in X^{++}\times X^{++}$, there is $T_1(u_0,v_0)>0$ such that
$$
u(t_0+t,x;t_0,u_0,v_0)\ge \eta_1
$$
for all $t\ge T_1(u_0,v_0)$, $x\in\RR$, and $t_0\in\RR$. This implies that
$$
v_t\le \mathcal{A} v+v(a_2(t,x)-b_2(t,x)\eta_1 -c_2(t,x) v),\quad x\in\RR
$$
for $t\ge T_1(u_0,v_0)$. It then follows that there is $\eta_2>0$ and $T_2(u_0,v_0)\ge T_1(u_0,v_0)$
such that
$$
v(t_0+t,x;t_0,u_0,v_0)\le v^*(t,x)-\eta_2,\quad x\in\RR
$$
for $t\ge T_2(u_0,v_0)$ and $t_0\in\RR$. Let $\eta=\min\{\eta_1,\eta_2\}$ and $T(u_0,v_0)=T_2(u_0,v_0)$, we have
$$
 u(t_0+t,x;t_0,u_0,v_0)\ge \eta,\,\,\, v(t+t_0,x;t_0,u_0,v_0)\le v^*(t+t_0,x)- \eta\quad \forall \,\, t\ge T(u_0,v_0),\,\, x\in\RR,\,\, t_0\in\RR.
 $$
This proves the corollary.
\end{proof}

\begin{remark}
 If $a_{1L}^0> \frac{c_{1M}^0a_{2M}^0}{c_{2L}^0}$, then $(0,v_0^*)$ is a linearly unstable solution of \eqref{main-eq1}, and if
 $a_{2L}^0> \frac{a_{1M}^0b_{2M}^0}{b_{1L}^0}$, then  $(u_0^*,0)$ is a linearly unstable solution of \eqref{main-eq1}.
\end{remark}

\section{Spreading Speeds}\label{spreading-speeds-section}

In this section, we investigate the invasion speed of the species $u$ to the species $v$ of  \eqref{main-eq1} and \eqref{main-eq2}.  Throughout this section,
 we assume that {\bf (H1)} holds.

 \subsection{Notations, definitions, and statements}

 In this subsection, we introduce some standing notions, definition of spreading speeds, and state the main results on spreading speeds.

 By (H1), $(0,v_0^*)$ is an unstable solution of \eqref{main-eq1} and $(u_0^*,0)$ is a globally stable solution of \eqref{main-eq1}.
   By Proposition \ref{semi-trivial-solu-prop} and Theorem \ref{stability-semitrivial-solu-thm},
$(0,v^*)$ is an unstable solution of \eqref{main-eq2} and
$$
\lim_{|x|\to \infty}|v^*(t,x)-v_0^*(t)|=0
$$
uniformly in $t\in\RR$.

To study the spreading speeds of \eqref{main-eq2}, we make the following standard change of variables,
\begin{equation}
\label{change-variable-eq2}
\tilde u=u,\quad \tilde v=v^*(t,x)-v.
\end{equation}
Dropping the tilde, \eqref{main-eq2} is transformed into
\begin{equation}
\label{main-eq2-1}
\begin{cases}
u_t=\mathcal{A}u+u\big(a_1(t,x)-b_1(t,x)u-c_1(t,x)(v^*(t,x)-v)\big),\quad x\in\RR\cr
v_t=\mathcal{A}v+b_2(t,x)\big(v^*(t,x)-v\big)u\cr
\qquad\qquad +v\big(a_2(t,x)-2c_2(t,x)v^*(t,x)+c_2(t,x)v\big),\quad x\in\RR.
\end{cases}
\end{equation}
Observe  that the trivial solution $E_0:=(0,0)$ of \eqref{main-eq2} becomes $\tilde E_0=(0,v^*)$,
the semitrivial solution $E_1:=(0,v^*)$ of \eqref{main-eq2} becomes $\tilde E_1=(0,0)$, and the semitrivial
solution $E_2:=(u^*,0)$ of \eqref{main-eq2} becomes $\tilde E_2=(u^*,v^*)$.
Note that $E_1$ is an unstable solution of \eqref{main-eq2-1}.

Consider \eqref{main-eq2-1}. For given $(u_0,v_0)\in X^+\times X^+$, we denote $(u(t,x;u_0,v_0),v(t,x;u_0,v_0))$ as
the solution of \eqref{main-eq2-1} with $(u(0,x;u_0,v_0),v(0,x;u_0,v_0))=(u_0(x),v_0(x))$.

By Proposition \ref{comparison-cooperative-prop},
we have the following lemma.

\begin{lemma}
\label{spreading-lm1}
For given $(u_1,v_1)$, $(u_2,v_2)\in X^+\times X^+$, if $0\le u_1\le u_2\le u^*(0,\cdot)$ and $0\le v_1\le v_2\le v^*(0,\cdot)$, then
$$
0\le u(t,\cdot;u_1,v_1)\le u(t,\cdot;u_2,v_2)\le u^*(t,x),\quad 0\le v(t,\cdot;u_1,v_1)\le v(t,\cdot;u_2,v_2)\le v^*(t,x)
$$
for all $t\ge 0$ and $x\in\RR$.
\end{lemma}

Let
\begin{align*}
X_1^+=\{u\in X^+\,|&\,  u(\cdot)< u^*(0,\cdot),\,\, u(x)=0\,\, {\rm for}\,\, x\gg 1,\\
&\,\, \liminf_{x\to -\infty}u(x)>0,\,\, \liminf_{x\to -\infty} (u^*(0,x)-u_0(x))>0\}
\end{align*}
and
\begin{align*}
X_2^+=\{v\in X^+\,|&\,  v(\cdot)< v^*(0,\cdot),\,\,  v(x)=0\,\, {\rm for}\,\, x\gg 1,\\
&\,\,\liminf_{x\to -\infty}v(x)>0,\,\, \liminf_{x\to -\infty} (v^*(0,x)-v_0(x))>0\}.
\end{align*}

\begin{definition}
\label{spreading-speed-cooperative}
Let
\begin{align*}
C_{\sup}=\big\{c\in\RR\,|\, &\limsup_{x\ge ct,t\to\infty} u^2(t,x;u_0,v_0)+v^2(t,x;u_0,v_0)=0,\\
&\,\,\,\forall\,\,
(u_0,v_0)\in X_1^+\times X_2^+\big\}
\end{align*}
and
\begin{align*}
C_{\inf}=\big\{c\in\RR\,|\, &\liminf_{x\le ct,t\to\infty}\min\{u(t,x;u_0,v_0), v(t,x;u_0,v_0)\}>0,\\
&\,\,\forall\,\,
(u_0,v_0)\in X_1^+\times X_2^+\big\}.
\end{align*}
Let
$$
c_{\sup}^*=\begin{cases} \inf\{c\,|\, c\in C_{\sup}\}\quad &{\rm if}\quad C_{\sup}\not=\emptyset\cr
\infty\quad &{\rm if}\quad C_{\sup}=\emptyset
\end{cases}
$$
and
$$
c_{\inf}^*=\begin{cases} \sup\{c\,|\, c\in C_{\inf}\}\quad &{\rm if}\quad C_{\inf}\not=\emptyset\cr
-\infty\quad &{\rm if}\quad C_{\inf}=\emptyset.
\end{cases}
$$
$[c_{\inf}^*,c_{\sup}^*]$ is called the {\rm spreading speed interval} of \eqref{main-eq2-1} or \eqref{main-eq2}.
\end{definition}

Before we state the main results on the spreading speeds of \eqref{main-eq1} and \eqref{main-eq2}, we
 recall the following proposition  proved  in \cite{KoRaSh}. To this end,
 we consider \eqref{main-eq1} and also
make the following standard change of variables,
\begin{equation*}
\tilde u=u,\quad \tilde v=v_0^*(t)-v.
\end{equation*}
Dropping the tilde, \eqref{main-eq1} is transformed into
\begin{equation}
\label{main-eq1-1}
\begin{cases}
u_t=\mathcal{A}u+u\big(a_1^0(t)-b_1^0(t)u-c_1^0(t)(v_0^*(t)-v)\big),\quad x\in\RR\cr
v_t=\mathcal{A}v+b_2^0(t)\big(v^*_0(t)-v\big)u+v\big(a_2^0(t)-2c_2^0(t)v^*_0(t)+c_2^0(t)v\big),\quad x\in\RR.
\end{cases}
\end{equation}

\begin{proposition}
\label{unperturbed-prop}
Consider \eqref{main-eq1}, i.e. \eqref{main-eq2} with $a_i=a_i^0$, $b_i=b_i^0$, and $c_i=c_i^0$ for $i=1,2$.
\begin{itemize}
 \item[(1)]  $c_{\inf}^*=c_{\sup}^*=c_0^*$. For  any $c<c_0^*(a_1^0,b_1^0,c_1^0,a_2^0,b_2^0,c_2^0)$, and $(u_0,v_0)\in X_1^+\times X_2^+$
 with $u_0(x)<u_0^*(0)$ and $v_0(x)<v_0^*(0)$, let
$(u^0(t,x;u_0,v_0),v^0(t,x;u_0,v_0))$ be
the solution of \eqref{main-eq1-1} with    $(u^0(0,x;u_0,v_0),v^0(0,x;u_0,v_0))=(u_0(x),v_0(x))$. Then
$$
\limsup_{x\le ct, t\to\infty} (u_0^*(t)-u^0(t,x;u_0,v_0))=0,\quad \limsup_{x\le ct, t\to\infty}(v_0^*(t)-v^0(t,x;u_0,v_0))=0.
$$

\item[(2)] Assume that (H2) holds. Then
$$
c_0^*=\inf_{\mu>0} \frac{\lambda(\mu,a_1^0-b_1^0 v_0^*)}{\mu}.
$$
\end{itemize}
\end{proposition}

We call $c_0^*$ the {\it spreading speed} of \eqref{main-eq1}. To indicate the dependence of $c_0^*$ on the coefficients of \eqref{main-eq1}, we
may write it as $c_0^*(a_1^0,b_1^0,c_1^0,a_2^0,b_2^0,c_2^0)$.
We now state the mains results on the spreading speeds of \eqref{main-eq1} and \eqref{main-eq2}.
The first theorem is on the continuity of the spreading speed of \eqref{main-eq1} with respect to spatially homogeneous time periodic perturbations.

\begin{theorem}
\label{main-thm0}
Consider \eqref{main-eq1}.
\begin{itemize}
\item[(1)] Assume that $\{(a_1^n,b_1^n,c_1^n,a_2^n,b_2^n,c_2^n)\}$ is a sequence of $T$-periodic H\"older continuous  positive functions
and
$$
\lim_{n\to\infty} \big(|a_i^n(t)-a_i^0(t)|+|b_i^n(t)-b_i^0(t)|+|c_i^n(t)-c_i^0(t)|\big)=0,\quad i=1,2.
$$
Then
$$
\liminf_{n\to\infty} c_0^*(a_1^n,b_1^n,c_1^n,a_2^n,b_2^n,c_2^n)\ge c_0^*(a_1^0,b_1^0,c_1^0,a_2^0,b_2^0,c_2^0).
$$

\item[(2)] Assume that  $\{(a_1^n,b_1^n,c_1^n,a_2^n,b_2^n,c_2^n)\}$ is as in (1) and (H2) holds. Then
$$
\lim_{n\to\infty} c_0^*(a_1^n,b_1^n,c_1^n,a_2^n,b_2^n,c_2^n)=c_0^*(a_1^0,b_1^0,c_1^0,a_2^0,b_2^0,c_2^0).
$$
\end{itemize}
\end{theorem}

The second theorem is on the effect of  localized spatial
variations on the spreading speeds of \eqref{main-eq1}.

\begin{theorem}
\label{main-thm1}
Consider \eqref{main-eq2}.
\begin{itemize}
\item[(1)]
Assume (H1). Then $c_{\sup}^*\ge c_{\inf}^*\ge c_0^*$.

\item[(2)]  Assume (H2).
 Then
$c_{\sup}^*=c_{\inf}^*=c_0^*.$
\end{itemize}
\end{theorem}

 We will prove Theorem \ref{main-thm0} and Theorem \ref{main-thm1} in next two subsections, respectively, and conclude this subsection with the following lemma, which will be used in the proofs of Theorems \ref{main-thm0} and \ref{main-thm1}.

 \begin{lemma}
\label{spreading-lm2}
\begin{itemize}
\item[(1)] Given $(u_n,v_n)\in X^+\times X^+$  with $u_n\le u^*(0,x)$ and $v_n(x)\le v^*(0,x)$ for $n\ge 0$,
 if $(u_n(x),v_n(x))\to (u_0(x),v_0(x))$ as $n\to\infty$ uniformly in bounded  subsets of $\RR$, then
$$
(u(t,x;u_n,v_n),v(t,x;u_n,v_n))\to (u(t,x;u_0,v_0),v(t,x;u_0,v_0))
$$
as $n\to\infty$ uniformly in $t$ in bounded subsets of $\RR^+$,  and $x$ in bounded subsets of $\RR$.

\item[(2)] For given $c\in\RR$ and $(u_0,v_0)\in X_1^+\times X_2^+$,
if $\liminf_{x\le ct, t\to\infty} u(t,x;u_0,v_0)>0$, then for any $c^{'}<c$,
$$
\limsup_{x\le c^{'}t,t\to\infty}u(t,x;u_0,v_0)>0,\quad \liminf_{x\le c^{'}t, t\to\infty}v(t,x;u_0,v_0)>0.
$$
\end{itemize}
\end{lemma}

\begin{proof}
(1) It can be proved by the similar arguments as those in \cite[Lemma 3.2]{KoRaSh}.

(2) Let $\sigma=\liminf_{x\le ct, t\to\infty} u(t,x;u_0,v_0)$. Then there is $n_0\in\NN$ such that
\begin{equation}
\label{aaux-eq1}
u(t,x;u_0,v_0)\ge \sigma/2\quad \forall \, t\ge n_0T,\,\, x\le ct.
\end{equation}
Let $\tilde u_0\equiv \min\{ \sigma/2, u_0^*(0)/2\}$ and $\tilde v_0\equiv 0$. By Corollary \ref{persistence-cor}, there are $\eta>0$ and  $n_1\in\NN$ such that
\begin{equation}
\label{aaux-eq2}
u(t,x;\tilde u_0,\tilde v_0)\ge \eta,\quad v(t,x;\tilde u_0,\tilde v_0)\ge \eta\quad \forall\,\, t\ge n_1T,\,\, x\in\RR,
\end{equation}
and
\begin{equation}
\label{aaux-eq2-1}
u_0(t,x;\tilde u_0,\tilde v_0)\ge \eta,\quad v_0(t,x;\tilde u_0,\tilde v_0)\ge \eta\quad \forall\,\, t\ge n_1T,\,\, x\in\RR,
\end{equation}
where  $(u_0(t,x;\tilde u_0,\tilde v_0)$, $v_0(t,x;\tilde u_0,\tilde v_0))$
is the solution of \eqref{main-eq1-1} with $(u_0(0,x;\tilde u_0,\tilde v_0)$, $v_0(0,x;\tilde u_0,\tilde v_0))=(\tilde u_0(x),\tilde v_0(x))$.

We claim that there is $L_0>0$ such that for any $l\in\RR$,
\begin{equation}
\label{aaux-eq3}
u(n_1T,x;\tilde u_l,\tilde v_l)\ge \eta/2,\quad v(n_1T,x;\tilde u_l,\tilde v_l)\ge \eta/2\quad \forall\,\, x\le l-L_0,
\end{equation}
where $(\tilde u_l,\tilde v_l)\in X^+\times X^+$ with $\tilde u_l(x)\le \min\{\sigma/2,u_0^*(0)/2\}$ for $x\in\RR$, $\tilde u_l(x)=\min\{\sigma/2,u_0^*(0)/2\}$ for $x\le l-1$,
 $\tilde u_l(x)=0$ for $x\ge l$, and $\tilde v_l\equiv 0$.

 In fact, assume that the claim is not true. Then there are $L_n\to\infty$, and $x_n,l_n\in\RR$ such that
 $x_n\le l_n-L_n$ and
 \begin{equation}
 \label{new-aux-eq1}
 u(n_1T,x_n;\tilde u_{l_n},\tilde v_{l_n})<\eta/2,\quad {\rm or}\quad v(n_1T,x_n;\tilde u_{l_n},\tilde v_{l_n})<\eta/2,\quad \forall\,\, n\ge 1.
 \end{equation}
 If $\{x_n\}$ is a bounded sequence, then
 $$\liminf_{n\to\infty}l_n\ge \liminf_{n\to\infty}(L_n+x_n)=\infty,
 $$
 which implies that
 $$
 \lim_{n\to\infty} (\tilde u_{l_n}(x),\tilde v_{l_n}(x))=(\tilde u_0(x),\tilde v_0(x))
 $$
 locally uniformly in $x$ in bounded subsets of $\RR$. Thus, by (1) and \eqref{aaux-eq2},
 $$
 u(u_1T,x_n;\tilde u_{l_n},\tilde v_{l_n})\ge \eta/2,\quad {\rm and}\quad v(n_1T,x_n;\tilde u_{l_n},\tilde v_{l_n})\ge\eta/2\quad \forall\,\, n\gg 1,
 $$
 which contradicts to \eqref{new-aux-eq1}.

 Hence $\{x_n\}$ is unbounded. Without loss of generality, we assume that $\lim_{n\to\infty}x_n=\infty$.  Let
 $$
 (\tilde u_n(t,x),\tilde v_n(t,x))=(u(t,x+x_n;\tilde u_{l_n},\tilde v_{l_n}),v(t,x+x_n;\tilde u_{l_n},\tilde v_{l_n})).
 $$
 Then $(\tilde u_n(t,x),\tilde v_n(t,x))$ is the solution of \eqref{main-eq2-1} with $a_i(t,x)$, $b_i(t,x)$, $c_i(t,x)$ ($i=1,2$),
 and $v^*(t,x)$ being replaced by $a_i(t,x+x_n)$, $b_i(t,x+x_n)$, $c_i(t,x+x_n)$, and $v^*(t,x+x_n)$, respectively,
 and with $(\tilde u_n(0,x),\tilde v_n(0,x))=(\tilde u_{l_n}(x+x_n),\tilde v_{l_n}(x+x_n))$.
 Note that
 $$
 \lim_{n\to\infty} a_i(t,x+x_n)=a_i^0(t),\,\,\, \lim_{n\to\infty} b_i(t,x+x_n)=b_i^0(t),\,\, \lim_{n\to\infty} c_i(t,x+x_n)=c_i^0(t),
 $$
 and
 $$
 \lim_{n\to\infty} v^*(t,x+x_n)=v_0^*(t), \,\, \lim_{n\to\infty} (\tilde u_{l_n}(x+x_n),\tilde v_{l_n}(x+x_n))=(\tilde u_0(x),\tilde v_0(x))
 $$
 locally uniformly in $x\in\RR$. By  the arguments  in \cite[Lemma 3.2]{KoRaSh},
 $$
 \lim_{n\to\infty}(\tilde u_n(t,x),\tilde v_n(t,x))=(u_0(t,x;\tilde u_0,\tilde v_0),v_0(t,x;\tilde u_0,\tilde v_0))
 $$
 locally uniformly in $t$ in bounded subsets of $\RR^+$ and in $x$ in bounded subsets of $\RR$. Then by \eqref{aaux-eq2-1},
 \begin{equation}
 \label{new-aux-eq2}
 \begin{cases}
 u(n_1T,x_n;\tilde u_{l_n},\tilde v_{l_n})=\tilde u_n(n_1T,0)\ge \eta/2,\quad \forall\,\, n\gg 1\cr
 v(n_1T,x_n;\tilde u_{l_n},\tilde v_{l_n})=\tilde v_n(n_1T,0)\ge \eta/2,\quad \forall\,\, n\gg 1.
 \end{cases}
 \end{equation}
This contradicts to \eqref{new-aux-eq1} again.
Therefore, the claim holds.

Fix any $0<c^{'}< c$. Let $n^*\ge \max\{n_0,n_1\}$ be such that $cn^*T-c^{'}n^*T\ge L_0$. Then by \eqref{aaux-eq1} and \eqref{aaux-eq3}, we have
$$
u((n^*+kn_1)T,x;u_0,v_0)\ge \eta/2,\quad v((n^*+kn_1)T,x;u_0,v_0)\ge \eta/2\quad \forall\,\, k\ge 1,\,\, x\le c^{'}(n^*+(k-1)n_1)T.
$$
This implies that
$$
\limsup_{x\le c^{'}t,t\to\infty}u(t,x;u_0,v_0)>0,\quad \liminf_{x\le c^{'}t, t\to\infty}v(t,x;u_0,v_0)>0.
$$
This proves (2).
\end{proof}

\subsection{Spreading speeds for unperturbed systems and the proof of Theorem \ref{main-thm0}}

In this subsection, we present some properties of spreading speeds for the unperturbed system
 \eqref{main-eq1} and prove Theorem \ref{main-thm0}.

\begin{proof}[Proof of Theorem \ref{main-thm0}]   Let
$(0,v_n^*(t))$ be the semitrivial solution of \eqref{main-eq2} with $a_i=a_i^n$, $b_i=b_i^n$, and $c_i=c_i^n$.

(1) Fix any $c<c_0^*(a_1^0,b_1^0,c_1^0,a_2^0,b_2^0,c_2^0)$.  For any  $(u_0,v_0)\in X_1^+\times X_2^+$ satisfying  that
$u_0(x)<u_0^*(0)$ and $v_0(x)<v_0^*(0)$ for $x\in\RR$, and
 $u_0(x)=0$ and $v_0(x)=0$ for $x\ge 0$,  let $(u^0(t,x;u_0,v_0)$, $v^0(t,x;u_0,v_0))$ be
the solution of \eqref{main-eq1-1} with   $(u^0(0,x;u_0,v_0),v^0(0,x;u_0,v_0))=(u_0(x),v_0(x))$. Let
 $(u^n(t,x;u_0,v_0),v^n(t,x;u_0,v_0))$ be the solution of \eqref{main-eq1-1} with $a_i^0$, $b_i^0$,  $c_i^0$, and $v_0^*(t)$  being replaced
 by $a_i^n$, $b_i^n$,  $c_i^n$, and $v_n^*(t)$,  respectively ($i=1,2$), and  with $(u^n(0,x;u_0,v_0),v^n(0,x;u_0,v_0))=(u_0(x),v_0(x))$.

By Proposition \ref{unperturbed-prop},  for any $0<\epsilon\ll 1$ with $\sup u_0\le  u_0^*(0)-2\epsilon$ and $\sup v_0\le v_0^*(0)-2\epsilon$, there is $K\ge 1$ such that
$$
u^0(t,x;u_0,v_0)\ge u_0^*(t)-\epsilon,\quad v^0(t,x;u_0,v_0)\ge v_0^*(t)-\epsilon\quad \forall\,\,  x\le ct,\,\,\, t\ge KT.
$$
Note that  there is $N\ge 1$ such that for $n\ge N$,
$$
u^n(t,x;u_0,v_0)\ge u^0(t,x;u_0,v_0)-\epsilon,\quad v^n(t,x;u_0,v_0)\ge v^0(t,x;u_0,v_0)-\epsilon\quad \forall\,\,  x\in\RR,\,\, 0\le t\le KT
$$
and
$$
u_n^*(t)\le u_0^*(t)+\epsilon,\quad v_n^*(t)\le v_0^*(t)+\epsilon\quad \forall\,\, t\in\RR.
$$
It then follows that  for $n\ge N$,
\begin{equation}
\label{aux-comp-eq}
\sup u_0\le u_n^*(0)-\epsilon,\quad \sup v_0\le v_n^*(0)-\epsilon
\end{equation}
and
\begin{equation}
\label{spreading-eq1}
u^n(KT,x;u_0,v_0)\ge u_0^*(KT)-2\epsilon,\quad v^n(KT,x;u_0,v_0)\ge v_0^*(KT)-2\epsilon \quad \forall\,\, x\le cKT.
\end{equation}
This implies that  for $n\ge N$,
\begin{equation}
\label{homogeneous-eq1}
u^n(KT,x+cKT;u_0,v_0)\ge u_0(x),\quad v^n(KT,x+cKT;u_0,v_0)\ge v_0(x)\quad \forall\,\, x\in\RR.
\end{equation}

By \eqref{aux-comp-eq}, \eqref{homogeneous-eq1} and  Proposition \ref{comparison-prop},
$$
u^n(t+KT,x+cKT;u_0,v_0)\ge u^n(t,x;u_0,v_0),\quad v^n(t+KT,x+cKT;u_0,v_0)\ge v^n(t,x;u_0,v_0)$$
for all $t\ge 0$, $x\in\RR$, and $n\ge N$. This together with \eqref{spreading-eq1} implies that for $n\ge N$,
$$
u^n(2KT,x;u_0,v_0)\ge u_0^*(0)-2\epsilon,\quad v^n(2KT,x;u_0,v_0)\ge v_0^*(0)-2\epsilon\quad \forall\,\, x\le 2 cKT.
$$

Continuing the above process, we have that for $n\ge N$,
$$
u^n(kKT,x;u_0,v_0)\ge u_0^*(0)-2\epsilon,\quad v^n(kKT,x;u_0,v_0)\ge v_0^*(0)-2\epsilon\quad \forall\,\, k\ge 1, \,\, x\le kcKT.
$$
This implies that for $n\ge N$,
$$
\liminf_{x\le ct, t\to\infty} u^n(t,x;u_0,v_0)>0,\quad  \liminf_{x\le ct, t\to\infty} v^n(t,x;u_0,v_0)>0.
$$
Hence
$$
c_0^*(a_1^n,b_1^n,c_1^n,a_2^n,b_2^n,c_2^n)\ge c\quad \forall\,\, n\ge N.
$$
This implies that
$$
\liminf_{n\to\infty} c_0^*(a_1^n,b_1^n,c_1^n,a_2^n,b_2^n,c_2^n)\ge c_0^*(a_1^0,b_1^0,c_1^0,a_2^0,b_2^0,c_2^0).
$$

(2) Note that (H0)-(H2) are also satisfied when $a_i^0$, $b_i^0$, and $c_i^0$ are replaced by $a_i^n$, $b_i^n$, and $c_i^n$, respectively
($i=1,2$) for $n\gg 1$. Let $(u_n^*(t),0)$ and $(0,v_n^*(t))$ be the semitrivial solutions of \eqref{main-eq1} with $a_i^0$, $b_i^0$, and $c_i^0$ being replaced by $a_i^n$, $b_i^n$, and $c_i^n$, respectively
($i=1,2$). Then
$$c_0^*(a_1^n,b_1^n,c_1^n,a_2^n,b_2^n,c_2^n)=\inf_{\mu>0}\frac{\lambda(\mu,a_2^n-b_2^n v_n^*)}{\mu}\quad \forall\,\, n\gg 1.
$$
Note also that $\lambda(\mu,a_2^n-b_2^n v_n^*)\to \lambda(\mu,a_2^0-b_2^0v_0^*)$ as $n\to\infty$ uniformly in $\mu$ in bounded sets. It then follows
that
$$
\lim_{n\to\infty}\inf_{\mu>0}\frac{\lambda(\mu,a_2^n-b_2^n v_n^*)}{\mu}=\inf_{\mu>0}\frac{\lambda(\mu,a_2^0-b_2^0 v_0^*)}{\mu}.
$$
Hence
$$
\lim_{n\to\infty} c_0^*(a_1^n,b_1^n,c_1^n,a_2^n,b_2^n,c_2^n)=c_0^*(a_1^0,b_1^0,c_1^0,a_2^0,b_2^0,c_2^0).
$$
\end{proof}

\subsection{Spreading speeds for perturbed systems}

In this subsection, we investigate the spreading speeds for the perturbed system \eqref{main-eq2} and prove Theorem \ref{main-thm1}.

First, we prove  Theorem \ref{main-thm1} (1).

\begin{proof}[Proof of Theorem \ref{main-thm1} (1)]
 It suffices to prove that $c_{\inf}^*\ge c_0^*$. Let $u_0\in X_1^+$ and $v_0\in X_2^+$ be such that
$u_0(x)=v_0(x)=0$ for $x\ge 0$ and $u_0(x)\ll \frac{1}{2}\min\{\inf u^*(t,x),\inf u_0^*(t)\}$,
$v_0(x)\ll \frac{1}{2} \min\{\inf v^*(t,x),\inf v_0^*(t)\}$ for $x<0$.

First of all, consider
\begin{equation}
\label{single-eq1}
u_t=\mathcal{A}u+u\Big(a_1(t,x)-b_1(t,x)u-c_1(t,x)v^*(t,x)\Big),\quad x\in\RR.
\end{equation}
Let
$$
\tilde c^*=\inf_{\mu>0}\frac{\lambda(\mu,a_1^0-c_1^0 v_0^*)}{\mu}.
$$
Fix $0<c<c_0^*$ and $0<\tilde c<\min\{c,\tilde c^*\}$.
By \cite[Theorem 2.2]{KoSh},
$$
\liminf_{x\le \tilde c t,t\to\infty} u(t,x;u_0)>0,
$$
where $u(t,x;u_0)$ is the solution of \eqref{single-eq1} with $u(0,x;u_0)=u_0(x)$.
Then  by Lemma \ref{spreading-lm2} and its arguments,
 there is $T_0>0$ such that
\begin{equation}
\label{spreading-eq2-1}
u(t,x;u_0,v_0)\ge \sup u_0,\quad v(t,x;u_0,v_0)\ge \sup v_0\quad \forall\,\, t\ge T_0,\,\, x\le \tilde c t.
\end{equation}

Next, note that
$$
\lim_{L\to\infty} \big(a_i(t,x+L)-a_i^0(t))=0,\quad \lim_{L\to\infty}\big(b_i(t,x+L)-b_i^0(t)\big)=0,\quad \lim_{L\to\infty}\big(c_i(t,x+L)-c_i^0(t)\big)=0
$$
and
$$
\lim_{L\to\infty} \big(u^*(t,x+L)-u_0^*(t)\big)=0,\quad \lim_{L\to\infty} \big(v^*(t,x+L)-v_0^*(t)\big)=0
$$
uniformly in $t\in\RR$ and $x$ in bounded sets. This implies that
\begin{equation}
\label{spreading-eq3}
\begin{cases}
\lim_{L\to\infty} \big(u(t,x+L;u_0(\cdot-L),v_0(\cdot-L))-u^0(t,x;u_0,v_0)\big)=0\cr
 \lim_{L\to\infty} \big(v(t,x+L;u_0(\cdot-L),v_0(\cdot-L))-
v^0(t,x;u_0,v_0))=0
\end{cases}
\end{equation}
uniformly in $t$ in bounded sets of $[0,\infty)$ and $x$ in bounded sets.
Note also that
\begin{equation}
\label{spreading-eq4}
\lim_{x\le ct,t\to\infty}\Big(|u^0(t,x;u_0,v_0)-u_0^*(t)|+|v^0(t,x;u_0,v_0)-v_0^*(t)|\Big)=0.
\end{equation}
This implies that there is $K\in\NN$ such that
\begin{equation}
\label{spreading-eq4-1}
u^0(KT,x;u_0,v_0)>2u^*_0(KT)/3,\quad v^0(KT,x;u_0,v_0)> 2v^*_0(KT)/3\quad \forall\,\, x\le cKT.
\end{equation}

Now, by \eqref{spreading-eq3} and \eqref{spreading-eq4-1}, there is $L_0>0$ such that
\begin{equation*}
\begin{cases}
u(KT,x+L;u_0(\cdot-L),v_0(\cdot-L))>u^*_0(KT)/2,\quad \forall \, \, L\ge L_0,\,\, |x|\le cKT\cr
 v(KT,x+L;u_0(\cdot-L),v_0(\cdot-L))>v^*_0(KT)/2,\quad \forall\,\, L\ge L_0,\,\, |x|\le cKT.
 \end{cases}
\end{equation*}
Hence
\begin{equation}
\label{spreading-eq4-2}
\begin{cases}
u(KT,x;u_0(\cdot-L),v_0(\cdot-L))>u^*_0(KT)/2,\quad \forall \, \, L\ge L_0,\,\, |x-L|\le cKT\cr
 v(KT,x;u_0(\cdot-L),v_0(\cdot-L))>v^*_0(KT)/2,\quad \forall\,\, L\ge L_0,\,\, |x-L|\le cKT.
 \end{cases}
\end{equation}

Let $K_0\in\NN$ be such that $K_0T\ge \max\{T_0,L_0\}$.
By \eqref{spreading-eq4-2}, we have
\begin{equation}
\begin{cases}
\label{spreading-eq4-3}
u((K_0+K)T,x;u_0,v_0)\ge u(KT,x;u_0(\cdot-\tilde c K_0T),v_0(\cdot-\tilde c K_0T)\ge u^*(KT)/2\cr
v((K_0+K)T,x;u_0,v_0)\ge v(KT,x;u_0(\cdot-\tilde c K_0T),v_0(\cdot-\tilde c K_0T)\ge v^*(KT)/2
\end{cases}
\end{equation}
for all $|x-\tilde c K_0T|\le cKT$.
By \eqref{spreading-eq2-1}, we have
$$
\begin{cases}
u((K_0+K)T,x;u_0,v_0)\ge \sup u_0\quad \forall\,\, x\le \tilde cK_0T+\tilde cKT\cr
v((K_0+K)T,x;u_0,v_0)\ge \sup v_0 \quad \forall\,\, x\le \tilde c K_0T+\tilde c KT.
\end{cases}
$$
It then follows that
\begin{equation}
\label{spreading-eq4-4}
\begin{cases}
u((K_0+K)T,x;u_0,v_0)\ge \sup u_0\quad \forall\,\, x\le \tilde c K_0T+ cKT\cr
v((K_0+K)T,x;u_0,v_0)\ge \sup v_0\quad \forall\,\,  x\le \tilde  cK_0T+cKT.
\end{cases}
\end{equation}

By \eqref{spreading-eq4-2} and \eqref{spreading-eq4-4}, we have
\begin{align*}
u((K_0+2K)T,x;u_0,v_0)&= u(KT,x;u((K_0+K)T,\cdot;u_0,v_0),v((K_0+K)T,\cdot;u_0,v_0))\\
&\ge u(KT,x; u_0(\cdot-\tilde c K_0T-c i K T),v_0(\cdot-\tilde c K_0 T-c i KT))\\
&\ge \sup u_0\quad \forall \,\, |x-\tilde c i K_0 T|\le  c KT
\end{align*}
and
\begin{align*}
v((K_0+2K)T,x;u_0,v_0)&= v(KT,x;u((K_0+K)T,\cdot;u_0,v_0),v((K_0+K)T,\cdot;u_0,v_0))\\
&\ge v(KT,x; u_0(\cdot-\tilde c K_0T-c i K T),v_0(\cdot-\tilde c K_0 T-c i KT))\\
&\ge \sup v_0\quad \forall \,\, |x-\tilde c i K_0 T|\le  c KT
\end{align*}
for $i=0,1$.
It then follows that
\begin{equation}
\label{spreading-eq4-5}
\begin{cases}
u((K_0+2K)T,x;u_0,v_0)\ge \sup u_0\quad \forall\,\, \tilde cK_0T-cKT\le x\le \tilde c K_0T+2cKT\cr
v((K_0+2K)T,x;u_0,v_0)\ge \sup v_0\quad \forall\,\, \tilde cK_0T-cKT\le x\le \tilde c K_0T+2cKT.
\end{cases}
\end{equation}
This together with \eqref{spreading-eq2-1} implies that
\begin{equation*}
\begin{cases}
u((K_0+2K)T,x;u_0,v_0)\ge \sup u_0\quad \forall \,\, x\le \tilde c K_0T+2cKT\cr
v((K_0+2K)T,x;u_0,v_0)\ge \sup v_0\quad \forall \,\, x\le \tilde  cK_0T+2cKT.
\end{cases}
\end{equation*}

By induction, we have
\begin{equation*}
\begin{cases}
u((K_0+nK)T,x;u_0,v_0)\ge \sup u_0\quad \forall\,\, x\le \tilde c K_0T+ncKT\cr
v((K_0+nK)T,x;u_0,v_0)\ge \sup v_0\quad \forall\,\, x\le \tilde  cK_0T+ncKT
\end{cases}
\end{equation*}
for all $n\ge 1$.
 It then follows that
$$
c_{\inf}^*\ge c\quad \forall\,\, c<c_0^*.
$$
This implies that $
c_{\inf}^*\ge c_0^*.$
\end{proof}

\begin{proof}[Proof of Theorem \ref{main-thm1} (2)]
Assume that $M_0>0$ is such that $a_i(t,x)=a_i^0(t)$, $b_i(t,x)=b_i^0(t)$,
and $c_i(t,x)=c_i^0(t)$ for $|x|\ge M_0$. By Theorem \ref{main-thm1} (1), it suffices to prove that
$c_{\sup}^*\le c_0^*$.

To prove $c_{\sup}^*\le c_0^*$, first of all,
 for given small $\epsilon>0$, consider the following perturbed system of \eqref{main-eq1-1}
\begin{equation}
\label{main-eq1-2}
\begin{cases}
u_t=\mathcal{A}u+u\big(a_1^\epsilon(t)-b_1^\epsilon(t)u-c_1^\epsilon(t)(v_0^*(t)-v)\big),\quad x\in\RR\cr
v_t=\mathcal{A}v+b_2^\epsilon(t)\big(v^*_0(t)-v\big)u+v\big(a_2^\epsilon(t)-2c_2^\epsilon(t)v^*_0(t)+c_2^\epsilon(t)v\big),\quad x\in\RR,
\end{cases}
\end{equation}
where $a_1^\epsilon(t)=a_1^0(t)+\epsilon$, $b_1^\epsilon(t)=b_1^0(t)-\epsilon$, $c_1^\epsilon(t)=c_1^0(t)-\epsilon$,
and $b_2^\epsilon(t)=b_2^0(t)+\epsilon$, $a_2^\epsilon(t)=a_2^0(t)+\epsilon+2\epsilon\sup v_0^*$, $c_2^\epsilon(t)=c_2^0(t)+\epsilon$.
Consider the linearization of \eqref{main-eq1-2} at $(0,0)$,
\begin{equation}
\label{main-eq1-4}
\begin{cases}
u_t=\mathcal{A}u+\big(a_1^\epsilon(t)-c_1^\epsilon(t)v_0^*(t))\big)u,\quad x\in\RR\cr
v_t=\mathcal{A}v+b_2^\epsilon(t)v^*_0(t)u+\big(a_2^\epsilon(t)-2c_2^\epsilon(t)v^*_0(t)\big)v,\quad x\in\RR.
\end{cases}
\end{equation}
For given $\mu>0$, let $\lambda_\epsilon(\mu)=\lambda(\mu,a_1^\epsilon-c_1^\epsilon v_0^*)$ and $\mu_\epsilon^*>0$ be such that
$$
\frac{\lambda_\epsilon(\mu_\epsilon^*)}{\mu_\epsilon^*}=\inf_{\mu>0}\frac{\lambda_\epsilon(\mu)}{\mu}.
$$
Let $c_\epsilon^*=\frac{\lambda_\epsilon(\mu_\epsilon^*)}{\mu_\epsilon^*}$. Note that $c_\epsilon^*\ge c_0^*$ and
$c_\epsilon^*\to c_0^*$ as $\epsilon\to 0$. To prove $c_{\sup}^*\le c_0^*$, it then suffices to prove $c_{\sup}^*\le c_\epsilon^*$ for any given
$0<\epsilon\ll 1$.

 Next, by (H2),
 \begin{equation}
 \label{claim-eq0}
 \begin{cases}
 a^\epsilon_1(t)-c^\epsilon_1(t)\frac{a^0_{2M}}{c^0_{2L}}-a^\epsilon_2(t)+2c^\epsilon_2 (t)\frac{a^0_{2L}}{c^0_{2M}}
-b^\epsilon_2(t)\frac{a^0_{2M}}{c^0_{2L}}\frac{c^\epsilon_{1M}}{b^\epsilon_{1L}}> 0\cr
 a^\epsilon_1(t)-c^\epsilon_1(t)\frac{a^0_{2M}}{c^0_{2L}}-a^\epsilon_2(t)+2c^\epsilon_2 (t)\frac{a^0_{2L}}{c^0_{2M}}-b^\epsilon_2(t)\frac{a^0_{2M}}{c^0_{2L}}\frac{c^\epsilon_{2M}}{b^\epsilon_{2L}}> 0
 \end{cases}
 \end{equation}
  for all $t\in\RR$ and $0<\epsilon\ll 1$. Fix $\epsilon>0$ such that \eqref{claim-eq0} holds.
  We prove $c_{\sup}^*\le c_\epsilon^*$. By Propositions \ref{principal-spectrum-prop1} and \ref{nonhomogeneous-prop},
  there are positive $T$-periodic functions $\phi_\epsilon^*(t)$ and
$\psi_\epsilon^*(t)$ such that
$(u,v)=(e^{-\mu_\epsilon^*(x-c_\epsilon^*t)}\phi_\epsilon^*(t),e^{-\mu_\epsilon^*(x-c_\epsilon^*t)}\psi_\epsilon^*(t))$ is a solution of \eqref{main-eq1-4}. Let
$$
u^+(t,x)=Ke^{-\mu_\epsilon^*(x-c_\epsilon^*t)}\phi_\epsilon^*(t),\quad
v^+(t,x)=Ke^{-\mu_\epsilon^*(x-c_\epsilon^*t)}\psi_\epsilon^*(t),
$$
where $K$  is a positive constant to be determined later.
We claim that
\begin{equation}
\label{main-eq1-5}
c_1^\epsilon(t)v^+(t,x)\le b_1^\epsilon(t) u^+(t,x),\quad c_2^\epsilon(t)v^+(t,x)\le b_2^\epsilon(t)u^+(t,x).
\end{equation}
We first assume the claim \eqref{main-eq1-5} is true and finish the proof of $c_{\sup}^*\le c_\epsilon^*$ and then prove the claim.

To prove $c_{\sup}^*\le c_\epsilon^*$, it suffices to prove that there is $L^*>0$ such that
\begin{equation}
\label{main-eq1-5-1}
u(t,x;u_0,v_0)\le u^+(t,x),\quad v(t,x;u_0,v_0)\le v^+(t,x)\quad \forall\,\, t\ge 0,\, \, x\ge c_\epsilon^* t+L^*.
\end{equation}
In order to do so,
let
$$
M^*=\max\{\sup_{t\in\RR}u_0^*(t), \sup_{t\in\RR}v_0^*(t),\sup_{t,x\in\RR} u^*(t,x),\sup_{t,x\in\RR}v^*(t,x)\}
$$
and
$$
m^*=\inf_{t\in\RR}\frac{\psi_\epsilon^*(t)}{\phi_\epsilon^*(t)},
$$
and $k\in\NN$ be such that
$$
km^*\ge 1.
$$
Let $\xi^*(t;K)$ be defined by
$$
u^+(t,\xi^*(t;K))=k M^*.
$$
Then
\begin{equation}
\label{x-star-eq}
\xi^*(t;K)=c_\epsilon^* t-\frac{1}{\mu_\epsilon^*}\ln \Big(\frac{kM^*}{K\phi_\epsilon^*(t)}\Big),
\end{equation}
and
$$
v^+(t,\xi^*(t;K))=kM^*\frac{\psi_\epsilon^*(t)}{\phi_\epsilon^*(t)}\ge k m^* M^*\ge M^*.
$$
Note that
$$
u^+(t,x)\le kM^*, \quad v^+(t,x)\le kM^*\frac{\psi_\epsilon^*(t)}{\phi_\epsilon^*(t)}\quad \forall\,\, x\ge \xi^*(t;K)
$$
and
$$
u^+(t,x)\ge  kM^*, \quad v^+(t,x)\ge kM^*\frac{\psi_\epsilon^*(t)}{\phi_\epsilon^*(t)}\quad \forall\,\, x\le \xi^*(t;K).
$$
Let
$$
K^*=kM^* \cdot \sup_{t\in\RR} b_2^\epsilon(t)
$$
and $g_1(u)$  be  a nondecreasing Lipschitz continuous  function satisfying that
$$
g_1(v)=\begin{cases} v\quad {\rm for}\quad v\le M^*\cr
M^*\quad {\rm for}\quad v\ge M^*.
\end{cases}
$$
Then
\begin{equation}
\label{main-eq1-6}
b_2^\epsilon\frac{|v_0^*(t)-g_1(v^+(t,x))|-\big(v_0^*(t)-g_1(v^+(t,x))\big)}{2} u^+(t,x)-K^*|v_0^*(t)-g_1(v^+(t,x))|\le 0
\end{equation}
for $x\ge \xi^*(t;K)$.
Let
$$
F_\epsilon(t,u,v)=u\big(a_1^\epsilon(t)-b_1^\epsilon(t)u-c_1^\epsilon(t)(v^*_0(t)-g_1(v))\big)
$$
and
\begin{align*}
G_\epsilon(t,u,v)&=b_2^\epsilon(t)\big(v^*_0(t)-g_1(v)\big)u+v\big(a_2^\epsilon(t)-2c_2^\epsilon(t)v^*_0(t)+c_2^\epsilon(t)g_1(v)\big)\\
& \quad + b_2^\epsilon(t)\frac{|v^*_0(t)-g_1(v)|-(v^*_0(t)-g_1(v))}{2} u-K^* |v_0^*(t)-g_1(v)|.
\end{align*}
By \eqref{main-eq1-5} and \eqref{main-eq1-6}
\begin{equation}
\label{main-eq1-7}
\begin{cases}
u_t^+\ge \mathcal{A}u^+ + F_\epsilon(t,x,u^+,v^+),\quad x\ge \xi^*(t;K)\cr
v_t^+\ge \mathcal{A} v^+ +G_\epsilon(t,x,u^+,v^+),\quad x\ge \xi^*(t;K).
\end{cases}
\end{equation}

Let $u(t,x)=u(t,x;u_0,v_0)$, $v(t,x)=v(t,x;u_0,v_0)$. Let
$$
F(t,u,v)=u\big(a_1(t)-b_1(t)u-c_1(t)(v^*(t,x)-g_1(v))\big)
$$
and
\begin{align*}
G(t,u,v)&=b_2(t)\big(v^*(t,x)-g_1(v)\big)u+v\big(a_2(t)-2c_2(t)v^*(t,x)+c_2(t)g_1(v)\big)\\
& \quad + b_2(t)\frac{|v^*(t,x)-g_1(v)|-(v^*(t,x)-g_1(v))}{2} u-K^* |v^*(t,x)-g_1(v)|.
\end{align*}
Note that
$$
u(t,x)\le u^*(t,x),\quad v(t,x)\le v^*(t,x)\quad \forall \,\, t\ge 0,\,\, x\in\RR.
$$
Hence
\begin{equation}
\label{main-eq2-2-1}
\begin{cases}
u_t = \mathcal{A} u+F(t,x,u,v),\quad x\in\RR\cr
v_t = \mathcal{A}v+G(t,x,u,v), \quad x\in\RR.
\end{cases}
\end{equation}

Note  that
$$
\lim_{K\to\infty} \xi^*(t;K)=\infty
$$
uniformly in $t\ge 0$ and
$$
\lim_{|x|\to \infty} |u^*(t,x)-u_0^*(t)|=\lim_{|x|\to\infty}|v^*(t,x)-v_0^*(t)|=0
$$
uniformly in $t\in\RR$. We can then
choose  $K\gg 1$ such that
$$
u_0(x)\le u^+(0,x),\quad v_0(x)\le v^+(0,x)\quad \forall\,\, x\in\RR,
$$
and
$$
\begin{cases}
F(t,x,u(t,x),v(t,x))\le F_\epsilon(t,x,u(t,x),v(t,x)),\quad \forall\,\, t\ge 0,\,\, x\ge \xi^*(t;K)\cr
 G(t,x,u(t,x),v(t,x))\le G_\epsilon(t,x,u(t,x),v(t,x))\quad \forall\,\, t\ge 0,\quad x\ge \xi^*(t;K).
 \end{cases}
$$
The last two inequalities  together with \eqref{main-eq2-2-1} imply that
\begin{equation}
\label{main-eq2-2}
\begin{cases}
u_t(t,x)\le \mathcal{A} u(t,x)+F_\epsilon(t,x,u(t,x),v(t,x)),\quad x\ge \xi^*(t;K)\cr
v_t(t,x) \le \mathcal{A}v(t,x)+G_\epsilon(t,x,u(t,x),v(t,x)),\quad x\ge \xi^*(t;K).
\end{cases}
\end{equation}

Note that
$$
u^+(t,x)\ge M^*,\quad v^+(t,x)\ge M^* \quad \forall\,\, x\le \xi^*(t;K),
$$
and
$$
u(t,x)\le M^*,\quad v(t,x)\le M^*\quad \forall\,\, x\le \xi^*(t;K).
$$
Then by \eqref{main-eq1-6}, \eqref{main-eq2-1}, and Proposition \ref{comparison-cooperative-prop},
$$
u(t,x)\le u^+(t,x),\quad v(t,x)\le v^+(t,x)\quad \forall\,\, t\ge 0,\,\, x\ge \xi^*(t;K).
$$
This implies that
$$
c_{\rm sup}^*\le c_\epsilon^*.
$$
Then by Theorem \ref{main-thm1} (1), we have $c_{\inf}^*=c_{\sup}^*=c_0^*$.

We return to prove the claim \eqref{main-eq1-5} now.    By the definition of $u^+(t,x)$ and $v^+(t,x)$, it suffices to prove
\begin{equation}
\label{main-eq1-5-new}
c_1^\epsilon(t)\psi^*(t)\le b_1^\epsilon(t)\phi^*(t),\quad c_2^\epsilon(t) \psi^*(t)\le b_2^\epsilon(t)\phi^*(t)\quad \forall\,\, t\in\RR.
\end{equation}
Observe that $(\phi^*(t),\psi^*(t))$ satisfying the following system
\begin{equation}
\label{claim-eq1}
\begin{cases}
u_t=\mathcal{A}(\mu_\epsilon^*)u-\lambda(\mu_\epsilon^*)u+\big(a_1^\epsilon(t)-c_1^\epsilon(t)v_0^*(t))\big)u,\quad x\in\RR\cr
v_t=\mathcal{A}(\mu_\epsilon^*)v-\lambda(\mu_\epsilon^*) v+b_2^\epsilon(t)v^*_0(t)u+\big(a_2^\epsilon(t)-2c_2^\epsilon(t)v^*_0(t)\big)v,\quad x\in\RR.
\end{cases}
\end{equation}
Hence $v=\psi^*(t)$ is a positive periodic solution of
\begin{equation}
\label{claim-eq2}
v_t=\mathcal{A}(\mu_\epsilon^*)v-\lambda(\mu_\epsilon^*) v+b_2^\epsilon(t)v^*_0(t)\phi^*(t)+\big(a_2^\epsilon(t)-2c_2^\epsilon(t)v^*_0(t)\big)v,\quad x\in\RR.
\end{equation}
We show that both $v=\frac{b_{1L}^\epsilon}{c_{1M}^\epsilon}\phi^*(t)$ and  $v=\frac{b_{2L}^\epsilon}{c_{2M}^\epsilon}\phi^*(t)$ are super-solutions of \eqref{claim-eq2}. In fact,
\begin{align*}
&\Big(\frac{b_{1L}^\epsilon}{c_{1M}^\epsilon}\phi^*(t)\Big)_t-\mathcal{A}(\mu_\epsilon^*)\frac{b_{1L}^\epsilon}{c_{1M}^\epsilon}\phi^*(t)
+\lambda(\mu_\epsilon^*) \frac{b_{1L}^\epsilon}{c_{1M}^\epsilon}\phi^*(t)-b_2^\epsilon(t)v^*_0(t)\phi^*(t)- \big(a_2^\epsilon(t)-2c_2^\epsilon(t)v^*_0(t)\big)\frac{b_{1L}^\epsilon}{c_{1M}^\epsilon}\phi^*(t)\\
&=(a_1^\epsilon(t)-c_1^\epsilon(t)v_0^*(t))\frac{b_{1L}^\epsilon}{c_{1M}^\epsilon}\phi^*(t)-b_2^\epsilon(t)v^*_0(t)\phi^*(t)- \big(a_2^\epsilon(t)-2c_2^\epsilon(t)v^*_0(t)\big)\frac{b_{1L}^\epsilon}{c_{1M}^\epsilon}\phi^*(t)\\
&=\frac{b_{1L}^\epsilon}{c_{1M}^\epsilon}\phi^*(t)\Big[ (a_1^\epsilon(t)-c_1^\epsilon(t)v_0^*(t))-b_2^\epsilon(t)v^*_0(t)\frac{c_{1M}^\epsilon}{b_{1L}^\epsilon}- \big(a_2^\epsilon(t)-2c_2^\epsilon(t)v^*_0(t)\big) \Big]\\
&\ge \frac{b_{1L}^\epsilon}{c_{1M}^\epsilon}\phi^*(t)\Big[ a_1^\epsilon(t)-c_1^\epsilon(t)\frac{a_{2M}^0}{c_{2L}^0}-b_2^\epsilon(t)\frac{a_{2M}^0}{c_{2L}^0}\frac{c_{1M}^\epsilon}{b_{1L}^\epsilon}- a_2^\epsilon(t)+2c_2^\epsilon(t)\frac{a_{2L}^0}{c_{2M}^0} \Big]\\
&>0\qquad\qquad \text{(by \eqref{claim-eq0})}.
\end{align*}
Hence $v=\frac{b_{1L}^\epsilon}{c_{1M}^\epsilon}\phi^*(t)$ is super-solution of \eqref{claim-eq2}.
Similarly, we have that
 $v=\frac{b_{2L}^\epsilon}{c_{2M}^\epsilon}\phi^*(t)$ is super-solution of \eqref{claim-eq2}.  By Proposition \ref{nonhomogeneous-prop},
 $v=\psi^*(t)$ is globally stable with respect to any perturbation $v_0(x)\equiv $constant. We then have
 $$
\psi^*(t)\le \frac{b_{1L}^\epsilon}{c_{1M}^\epsilon}\phi^*(t),\quad \psi^*(t)\le \frac{b_{2L}^\epsilon}{c_{2M}^\epsilon}\phi^*(t),
$$
which implies \eqref{claim-eq1} and then \eqref{main-eq1-5}.
\end{proof}

\end{document}